\documentclass[11pt]{article}

\usepackage{latexsym}
\usepackage{amssymb}
\usepackage{amsthm}
\usepackage{amscd}
\usepackage{amsmath}
\usepackage{mathrsfs}
\usepackage[all]{xy}
\usepackage{hyperref} 
\usepackage[usenames,dvipsnames]{color}
\usepackage{graphicx,eepic}
\usepackage{float}

\usepackage[small,nohug,heads=vee]{diagrams} 
\diagramstyle[labelstyle=\scriptstyle]

\usepackage{color}

\theoremstyle{definition}
\newenvironment{letterthm}[2][Theorem]{\begin{trivlist}
\item[\hskip \labelsep {\bfseries #1}\hskip \labelsep {\bfseries \hspace{-0.5mm}#2.}]}{\end{trivlist}}

\newcommand{\descref}[1]{\hyperref[#1]{#1}}

\theoremstyle{definition}
\newtheorem* {theorem*}{Theorem}
\newtheorem{theorem}{Theorem}[subsection]

\newtheorem* {problem*}{Problem}
\theoremstyle{definition}

\newtheorem* {example*}{Example}
\newtheorem{fact}[theorem]{Fact}
\newtheorem{lemma}[theorem]{Lemma}
\theoremstyle{definition}
\newtheorem{definition}[theorem]{Definition}
\theoremstyle{definition}
\newtheorem* {notation}{Notation}

\newtheorem{proposition}[theorem]{Proposition}
\newtheorem{corollary}[theorem]{Corollary}

\newtheorem* {remark}{Remark}
\theoremstyle{definition}
\newtheorem {example}[theorem]{Example}
\theoremstyle{definition}

\theoremstyle{definition}

\theoremstyle{definition}

\xyoption{dvips}

\usepackage{fullpage}

\numberwithin{equation}{section}

\def\qquand{\qquad\text{and}\qquad}

\def\({\left(}
\def\){\right)}
     \newcommand{\CC}{\mathbb{C}}  \newcommand{\QQ}{\mathbb{Q}}   \newcommand{\cP}{\mathcal{P}} 
 \newcommand{\cK}{\mathcal{K}} \newcommand{\cO}{\mathcal{O}} 
  \newcommand{\cS}{\mathcal{S}}

                \def\spanning{\textnormal{-span}}   
  \def\wt{\widetilde}
 \def\L{\mathcal{L}} \def\sh{\mathrm{sh}}

\newcommand{\h}{\mathfrak{h}}
\newcommand{\fn}

\def\fk{\mathfrak}

\def\barr{\begin{array}}
\def\earr{\end{array}}
\def\ba{\begin{aligned}}
\def\ea{\end{aligned}}
\def\be{\begin{equation}}
\def\ee{\end{equation}}
\def\cS{\mathcal{S}}

\def\cF{\mathcal{F}}

\def\ben{\begin{enumerate}}
\def\een{\end{enumerate}}

\def\omdef{\overset{\mathrm{def}}}
\def\KK{\mathbb{K}}

\newcommand{\tthree}[6]{\xy<0.0cm,0.25cm> \xymatrix@R=0.5cm@C=0.2cm{ *{\bullet} #6 &*{\bullet} #5 &*{\bullet} #4 \\  *{\bullet}  #1 &*{\bullet}  #2 &*{\bullet}  #3  }\endxy}

\newcommand{\tfour}[8]{\xy<0.0cm,0.25cm> \xymatrix@R=0.2cm@C=0.2cm{ *{\bullet} #8 & *{\bullet} #7 &*{\bullet} #6 &*{\bullet} #5 \\  *{\bullet}  #1 &*{\bullet}  #2 &*{\bullet}  #3  &*{\bullet}  #4  }\endxy}

\newcommand{\exone}{\xy<0cm,0cm> \xymatrix@R=0.0cm@C=0.3cm{*{\bullet}1 &*{\bullet}  
}\endxy}

\newcommand{\extwo}[4]{\xy<0cm,0cm> \xymatrix@R=-0.0cm@C=0.3cm{*{\bullet} #1 &*{\bullet} #2 &*{\bullet} #3 &*{\bullet}  #4  
}\endxy}

\newcommand{\exthree}[6]{\xy<0cm,0cm> \xymatrix@R=-0.0cm@C=0.3cm{*{\bullet} #1 &*{\bullet} #2 &*{\bullet} #3 &*{\bullet}  #4 &*{\bullet}  #5 &*{\bullet}  #6
 }\endxy}

\newcommand{\exfour}[8]{\xy<0.0cm,0.0cm> \xymatrix@R=-0.0cm@C=0.3cm{*{\bullet} #1 &*{\bullet} #2 &*{\bullet} #3 &*{\bullet}  #4 &*{\bullet}  #5 &*{\bullet}  #6 &*{\bullet}  #7 &*{\bullet}  #8
}\endxy}

\newcommand{\xone}[1]{\xy<0cm,0cm> \xymatrix@R=-0.0cm@C=0.3cm{*{\bullet} #1 
  }\endxy}

\newcommand{\xtwo}[2]{\xy<0cm,0cm> \xymatrix@R=-0.0cm@C=0.3cm{*{\bullet} #1 &*{\bullet} #2  
  }\endxy}

\newcommand{\xthree}[3]{\xy<0cm,0cm> \xymatrix@R=-0.0cm@C=0.3cm{*{\bullet} #1 &*{\bullet} #2 &*{\bullet} #3 
}\endxy}

\newcommand{\xfour}[4]{\xy<0.0cm,0.0cm> \xymatrix@R=-0.0cm@C=.3cm{*{\bullet} #1 &*{\bullet} #2 &*{\bullet} #3 &*{\bullet}  #4 
}\endxy}

\newcommand{\xfive}[5]{\xy<0.0cm,0.0cm> \xymatrix@R=-0.0cm@C=0.3cm{*{\bullet} #1 &*{\bullet} #2 &*{\bullet} #3 &*{\bullet}  #4  &*{\bullet}  #5 
}\endxy}

\newcommand{\xsix}[6]{\xy<0.0cm,0.0cm> \xymatrix@R=-0.0cm@C=0.3cm{*{\bullet} #1 &*{\bullet} #2 &*{\bullet} #3 &*{\bullet}  #4  &*{\bullet}  #5 &*{\bullet}  #6 
}\endxy}

\newcommand{\xseven}[7]{\xy<0.0cm,0.0cm> \xymatrix@R=-0.0cm@C=0.3cm{*{\bullet} #1 &*{\bullet} #2 &*{\bullet} #3 &*{\bullet}  #4  &*{\bullet}  #5 &*{\bullet}  #6 &*{\bullet}  #7 
}\endxy}

\def\kk{\mathbb{K}}

\def\id{\mathrm{id}}

\def\im{\mathrm{Image}}
\newarrow {Corresponds} <--->

\def\Set{\textbf{Set}}

\def\Sp{\textbf{Sp}}
\def\SetSp{\textbf{SetSp}}

\def\Hopf{\textbf{Hopf}}

\def\cocoHopf{{^{\mathrm{co}}\Hopf^{\mathrm{co}}}}

\def\Vec{\textbf{Vec}}
\def\h{\mathbf{h}}

\def\p{\mathbf{p}}
\def\q{\mathbf{q}}
\def\r{\mathbf{r}}

\def\X{\mathrm{X}}
\def\E{\mathrm{E}}
\def\L{\mathrm{L}}

\def\x{\mathbf{x}}

\def\R{\mathrm{R}}
\def\P{\mathrm{P}}
\def\Q{\mathrm{Q}}

\def\one{\mathbf{1}}

\def\bfE{\mbox{\textbf{E}}}
\def\bfX{\mbox{\textbf{X}}}
\def\bfL{\mbox{\textbf{L}}}

\def\bfPi{\mbox{\boldmath$\Pi$}}

\def\Sym{\mathbf{Sym}}
\def\NCSym{\mathbf{NCSym}}

\def\S{{\tt S}}

\def\FB{\mathbf{FB}}
\def\FI{\mathbf{FI}}

\makeatletter
\renewcommand{\@makefnmark}{\mbox{\textsuperscript{}}}
\makeatother

\allowdisplaybreaks[1]

\UseCrayolaColors

\begin{document}
\title{Strong forms of self-duality for Hopf monoids in species
}
\author{Eric Marberg\footnote{This research was conducted with support from the National Science Foundation.}
\\
Department of Mathematics \\
Stanford University \\
{\tt emarberg@stanford.edu}}
\date{}

\maketitle

\begin{abstract}
A vector species is a functor  from the category of finite sets with bijections to vector spaces (over a fixed field); informally, one can view this as a sequence of $S_n$-modules, one for each natural number $n$. 
A Hopf monoid (in the category of vector species) consists of a vector species with  unit, counit, product, and coproduct morphisms satisfying several compatibility conditions, analogous to a  graded  Hopf algebra.

A vector species has a basis if and only if it is given by a sequence of $S_n$-modules which are permutation representations.
We say that a Hopf monoid is {freely self-dual} if it is connected and finite-dimensional, and if it has a basis in which the structure constants of its product and coproduct coincide. Such Hopf monoids are self-dual in the usual sense, and we show that they are furthermore both commutative and cocommutative. 

We prove more specific classification theorems for freely self-dual Hopf monoids whose  products (respectively, coproducts) are linearized in the sense that they preserve the basis; we call such Hopf monoids strongly self-dual (respectively, linearly self-dual). In particular, we show that every strongly self-dual Hopf monoid has a basis isomorphic to some species of block-labeled set partitions, on which the product acts as the disjoint union. In turn, every linearly self-dual Hopf monoid has a basis isomorphic to the species of maps to a fixed set, on which the coproduct acts as restriction. 
It follows that every linearly self-dual Hopf monoid is strongly self-dual. 

Our final results concern connected Hopf monoids which are finite-dimensional, commutative, and cocommutative. We prove that such a Hopf monoid has a basis in which  its product and coproduct are both linearized if and only if it is strongly self-dual with respect to a basis equipped with a certain partial order, generalizing the refinement partial order on set partitions.
\end{abstract}

\setcounter{tocdepth}{2}
\tableofcontents

\section{Introduction}

\subsection{Background}

Fix a field $\kk$ with characteristic zero.
A \emph{vector species} is a functor $\p : \FB \to \Vec$
where $\FB$ denotes the category of finite sets with bijective maps as morphisms and $\Vec$ denotes the category of $\kk$-vector spaces.
A species $\p$ thus assigns to each finite set $S$ a vector space $\p[S]$ and to 
each bijection  $\sigma : S \to S'$ between finite sets a linear map $\p[\sigma] : \p[S]\to \p[S']$. 
The only condition which these assignments must satisfy is that
\be\label{functorial-eq} \p[\id_S] = \id_{\p[S]} \qquand \p[\sigma \circ \sigma'] = \p[\sigma]\circ \p[\sigma']\ee
 whenever $\sigma$ and $\sigma'$ are composable bijections between finite sets.
The composition rule implies that 
$\p[\sigma]$ is always a bijection, since $\p[\sigma^{-1}]$ affords its inverse. 
In this definition and subsequently, we follow the conventions of Aguiar and Mahajan's book \cite{species}.

Vector species form a category in which morphisms are natural transformations. Thus a morphism of vector species  $f: \p \to \q$  consists of a $\kk$-linear map $f_S : \p[S] \to \q[S]$ for each finite set  $S$, such that 
\[\q[\sigma]\circ f_S = f_{S'} \circ \p[\sigma]\]
for all bijections $\sigma : S \to S'$.
We refer to the map $f_S$ as the \emph{$S$-component} of $f$.
This category has a  symmetric monoidal structure which leads to interesting notions of monoids, comonoids, bimonoids, and Hopf monoids in species; see Section \ref{cauchy-sect}.
There are  natural functors which one can use to turn vector species into graded vector spaces, and Hopf monoids in species into graded Hopf algebras, as detailed in \cite[Chapter 15]{species}. 
Constructions and properties of Hopf algebras which lift to the species level sometimes become easier to understand and to generalize via these functors.

In this paper 
we analyze the structure of Hopf monoids in species exhibiting certain strong forms of self-duality,  analogous to Zelevinksy's notion of \emph{positive self-duality} for Hopf algebras \cite{zel}.
Our original motivation was to provide a simple way of constructing and understanding the relationships between various Hopf algebras attached to towers of unipotent groups, considered for example in \cite{AZ,ABT,CB}. 
The applications of our results towards this end (sketched at the end of this introduction) are the main topic of the supplementary work \cite{supp}.


 
To proceed,
we review briefly the definition of
a \emph{Hopf monoid} in species.
For us, this consists of a triple $\h = (\p,\nabla,\Delta)$, where $\p$ is a vector species equipped with a unit map
$  \kk \to \p[\varnothing]$ and a counit map $  \p[\varnothing] \to \kk$,
and where $\nabla = (\nabla_{S,T})$ and $\Delta = (\Delta_{S,T})$ are two systems of  linear maps 
\[ \nabla_{S,T} : \p[S]\otimes \p[T] \to \p[S\sqcup T]\qquand \Delta_{S,T} : \p[S\sqcup T] \to \p[S]\otimes \p[T]\]
indexed by pairs of disjoint finite sets $S$, $T$. (Here and throughout, we write $\sqcup$ to indicate the disjoint union of sets.) These systems, which we call the product and coproduct of $\h$, are functorial in the sense that the diagrams
\be\label{1st-diagrams}
\begin{diagram}
 \p[ S \sqcup T] && \rTo^{\p[\sigma]} && \p[S'\sqcup T'] \\
\uTo^{\nabla_{S,T}}&  &&& \uTo_{\nabla_{S',T'}} \\
\p[S] \otimes \p[T] & &\rTo^{ \p[\sigma|_S]\otimes \p[\sigma|_T] }  && \p[S']\otimes \p[T'] 
\end{diagram}
\qquad\qquad
\begin{diagram}
 \p[ S \sqcup T] && \rTo^{\p[\sigma]} && \p[S'\sqcup T'] \\
\dTo^{\Delta_{S,T}}&  &&& \dTo_{\Delta_{S',T'}} \\
\p[S] \otimes \p[T] &&\rTo^{ \p[\sigma|_S]\otimes \p[\sigma|_T] } && \p[S']\otimes \p[T'] 
\end{diagram}
\ee
commute for any bijection $\sigma : S\sqcup T \to S' \sqcup T'$ with $\sigma(S) = S'$ and $\sigma(T) = T'$. (We write $\otimes $ for the tensor product over $\kk$.)
For simplicity we always assume that $\h$ is \emph{connected}, by which we mean that the unit and counit are inverse $\kk$-linear bijections.
The product and coproduct must satisfy a few natural conditions; namely, they must be \emph{associative}, \emph{unital}, and \emph{Hopf compatible}. These conditions are defined precisely in Section \ref{monoidal-struct-sect}.

Informally, associativity means that the two obvious ways of defining each of the notations $\nabla_{R,S,T}$ and $\Delta_{R,S,T}$ 
coincide.
Unitality means that if  $S$ or $T$ is empty then $\nabla_{S,T}$ and $\Delta_{S,T}$ may be identified with (co)multiplication by elements of $\kk\cong \p[\varnothing]$.
If we think of the product and coproduct as ``joining'' and ``splitting'' operations, then Hopf compatibility is the condition that we get the same thing by joining then splitting two objects as by splitting two objects into four  then joining appropriate pairs; see Definition \ref{hopf-compat-def}.

\begin{example}\label{E-ex}
Let $C$ be any set.
Define $\bfE_C$ as the vector species such that if $S$ is a finite set then
 $\bfE_C[S]$ is the $\kk$-vector space whose basis is the set of  maps $f : S \to C$. If $\sigma : S \to S'$ is a bijection then let $\bfE_C[\sigma]$ act on the basis of maps by the formula $f\mapsto f\circ \sigma^{-1}$.
We view $\bfE_C$ as a connected Hopf monoid 
whose unit and counit are the inverse isomorphisms $\bfE_C[\varnothing] \cong \kk$ which identify the unique map $\varnothing \to C$ with $1_\kk \in \KK$,
and whose product and coproduct have the formulas
\[ \nabla_{S,T}(f\otimes g) = f\sqcup g\qquand \Delta_{S,T}(h) = h|_S \otimes h|_T\]
for maps $f : S \to C$ and $g : T \to C$ and $h : S\sqcup T \to C$.
Here $h|_S$ denotes the restriction of $h$ to $S$, while $f\sqcup g$ denotes the unique map $S\sqcup T \to C$ which restricts to $f$ on $S$ and to $g$ on $T$.
\end{example}

Up to isomorphism (where morphisms of Hopf monoids are natural transformations between vector species which commute with the unit, counit, product, and coproduct; see Section \ref{monoidal-struct-sect}), $\bfE_C$ depends only on the cardinality of $C$,
and we define the \emph{unit} and \emph{exponential species} as
\be\label{one-def}\one = \bfE_\varnothing
\qquand 
\bfE= \bfE_{\{1\}}.
\ee We  identify $\one[\varnothing ]  =\bfE[S] = \kk$ for all finite sets, while $\one[S] = 0$ if $S$ is nonempty. 
If $C$ is the basis of a $\kk$-vector space $V$ then $\bfE_C$ is the same as the Hopf monoid $\bfE_V$ defined in \cite[Example 8.18]{species}.

\begin{example}\label{Pi-ex}
Recall that a \emph{partition} of a set $S$ is a set of pairwise disjoint nonempty sets (referred to as \emph{blocks}) whose union is $S$. 
Define $\bfPi$ as the vector species such that $\bfPi[S]$ is the $\kk$-vector space whose basis 
is the set of partitions of $S$. If $\sigma : S \to S'$ is a bijection then we let $\bfPi[\sigma]$ act on the basis of set partitions by the formula $X \mapsto \{ \sigma(B) : B \in X\}$.
Then $\bfPi$ is a connected Hopf monoid whose unit and counit are the inverse isomorphisms $\bfPi[\varnothing]\cong \kk$  which identify the unique partition of the empty set with $1_\kk \in \kk$, and whose product and coproduct are
\[ \nabla_{S,T}(X\otimes Y) = X \sqcup Y \qquand \Delta_{S,T}(Z) = Z|_S \otimes Z|_T\]
for partitions $X$, $Y$,  $Z$ of the sets $S$, $T$,  $S\sqcup T$. Here $Z|_S$ denotes the partition of $S$ induced by the partition $Z$ of the larger set $S\sqcup T$.
Explicitly, $Z|S = \{ B \cap S : B \in Z \} - \{\varnothing\}$.
\end{example}

The products and coproducts of the Hopf monoids in the preceding examples have  all been (co)commutative in the usual sense; see Definitions \ref{assoc-def} and \ref{coassoc-def}.
In the following example the product is not commutative.

\begin{example}\label{L-ex}
Define $\bfL$ as the vector species such that $\bfL[S]$ is the $\kk$-vector space whose basis 
is the set of linear orders of $S$. If $\sigma : S \to S'$ is a bijection then we let $\bfL[\sigma]$ act on the basis of orders   by the formula $\ell \mapsto \ell'$ where $\ell'$ is the order of $S'$ such that $i < j$ with respect to $\ell'$ if and only if $\sigma^{-1}(i) < \sigma^{-1}(j)$ with respect to $\ell$.
Then $\bfL$ is a connected Hopf monoid whose unit and counit are given by identifying $\bfL[\varnothing]\cong \kk$ as usual, and whose product and coproduct have the formulas
\[ \nabla_{S,T}(\ell_1\otimes \ell_2) = \ell_1 * \ell_2 \qquand \Delta_{S,T}(\ell) = \ell|_S \otimes \ell|_T\]
for linear orders $\ell_1$, $\ell_2$,  $\ell$ of the sets $S$, $T$,  $S\sqcup T$. Here $\ell|_S$ denotes the linear order of $S$ induced by $\ell$, and $\ell_1*\ell_2$ denotes the unique linear order of $S\sqcup T$ restricting to $\ell_1$ on $S$ and to $\ell_2$ on $T$, such that every element of $T$ exceeds every element of $S$. 

\end{example}

As mentioned earlier, there are several ways of constructing   graded Hopf algebras from Hopf monoids in species; see, for the origins of some of these constructions, the papers  of Stover \cite{stover} and  Patras, Reutenauer, and Schocker \cite{291,292,293}.
 Among these methods are two functors
 $\cK$ and $\overline \cK$, 
 referred to as   \emph{Fock functors} in  \cite[Chapter 15]{species}. 
The images of a Hopf monoid $\h = (\p,\nabla,\Delta)$ 
under these functors are  Hopf algebra structures
on the graded vector spaces
\[\cK(\h) = \bigoplus_{n\geq 0} \p[n]
\qquand
\overline\cK(\h)=\bigoplus_{n\geq 0} \p[n]_{S_n}\] where $\p[n]_{S_n}$ denotes the quotient space of coinvariants in $\p[n]$ under the natural action of the symmetric group.
For the precise definitions of these Hopf algebras, see 
\cite[Theorem 15.12]{species}.

We  mention these functors  just to point out some of the familiar Hopf algebras which arise from the simple Hopf monoids given as examples above. To begin, $\cK(\bfE_C)$ and $\overline\cK(\bfE_C)$ are the  tensor and symmetric algebras on the $\kk$-vector space generated by $C$; see \cite[Example 15.16]{species}. In turn, $\cK(\bfPi)$ and $\overline{\cK}(\bfPi)$ coincide with the Hopf algebras 
$\NCSym$  (see \cite{preprint,Sagan1})
and $\Sym$  (see \cite{ReinerNotes,zel})
of symmetric functions in  countably many non-commuting and commuting variables; see \cite[Section 17.4]{species}. Applying $\cK$ to the dual of $\bfL$ (with the \emph{dual} of a Hopf monoid defined as in Section \ref{dual-sect}), gives the Hopf algebra of permutations introduced by Malvenuto and Reutenauer  \cite{Mal1,Mal2}; see \cite[Example 15.17]{species}. Many more  such identifications appear in \cite[Chapter 17]{species}; this provides one motivation for the abstract study of Hopf monoids in species.

\subsection{Outline} 

A \emph{set species} is a functor $\P : \FB \to \Set$ where $\Set$ denotes the usual category of sets
with arbitrary maps as morphisms. 
Thus, $\P$ assigns, to each finite set $S$ and bijection $\sigma: S \to S'$, a set $\P[S]$ and a bijection $\P[\sigma] : \P[S] \to \P[S']$, subject again  to  conditions  \eqref{functorial-eq}.
A set species $\P$ is \emph{finite} if $\P[S]$ is always finite; likewise, a vector species $\p$ is \emph{finite-dimensional} if $\p[S]$ is always finite-dimensional.
Following \cite[Section 8.7]{species}, if $\P$ is a set species then we write $\kk\P$ for the vector species  with $(\kk\P)[S]$ given by the vector space generated by $\P[S]$ and $(\kk\P)[\sigma]$ given by the vector space isomorphism which is the linearization of $\P[\sigma]$.

A   property present in each Hopf monoid $\h = (\p,\nabla,\Delta)$ in the examples above is that
the underlying vector species $\p$ is \emph{linearized}, by which we mean there exists a set species $\P$ such that $\p = \kk\P$. We refer to the set species $\P$ as a \emph{basis} for $\p$.
%
 Not every vector species has a basis (see Proposition \ref{basis-prop}) but when one exists 
we can define the corresponding \emph{structure constants} of the product and coproduct as
the elements $a_{\lambda',\lambda''}^\lambda$ and $b_{\lambda',\lambda''}^{\lambda}$ in $\kk$ such that if $S$, $T$ are disjoint finite sets then
\[ \nabla_{S,T}(\lambda' \otimes \lambda'') = \sum_{\lambda \in \P[S\sqcup T]} a_{\lambda',\lambda''}^\lambda\cdot \lambda
\qquand 
\Delta_{S,T}(\lambda) = \sum_{(\lambda',\lambda'') \in \P[S]\times \P[T]} b_{\lambda',\lambda''}^\lambda \cdot \lambda'\otimes \lambda''\]
for $\lambda \in \P[S\sqcup T]$ and $\lambda' \in \P[S]$ and $\lambda'' \in \P[T]$.
For the Hopf monoids in Examples \ref{E-ex} and \ref{Pi-ex}, there exist bases such that these structure constants coincide, that is, 
such that we always have $a_{\lambda',\lambda''}^\lambda = b_{\lambda',\lambda''}^{\lambda}$.
(Note that in Example \ref{Pi-ex}, the basis exhibiting this property is not  the obvious basis of set partitions.)
If $\h$ is finite-dimensional and it has this property, then
we say that it is 
 \emph{freely self-dual} (or \emph{FSD} for short) with respect to the basis $\P$; see Section \ref{fsd-sect} for an explanation of this terminology.
The main results of this work are organized around the problem of classifying connected Hopf monoids which are FSD or which exhibit related forms of self-duality.

We summarize the rest of this paper as follows.
Section \ref{prelim-sect} provides some additional preliminaries on species$-$in particular,  the precise definitions of monoids, comonoids, and Hopf monoids.
In Section \ref{BSD-sect} we consider FSD Hopf monoids in detail.
Our first main result is the following theorem, from which one can deduce  that the Hopf monoid of linear orders $\bfL$ is not FSD.

\begin{letterthm}{A}\label{A} (See Theorem \ref{coco-cor}.) Assume $\KK\subset \CC$. If a connected Hopf monoid is freely self-dual then it is both commutative and cocommutative.
\end{letterthm}

A \emph{subspecies} of a vector species $\p$ is a (set or vector) species $\q$ such that $\q[S]\subset\p[S]$ for all finite sets $S$ and such that $\q[\sigma]$ is the restriction of $\p[\sigma]$ for all bijections $\sigma$; to indicate this situation we write $\q \subset \p$.
The subspecies of \emph{primitive elements} of a connected Hopf monoid $\h = (\p,\nabla,\Delta)$ is the vector subspecies $\cP(\h)\subset\p$ with 
$\cP(\h)[\varnothing] = 0$ and 
\[ {\cP}(\h)[I] = \{ \lambda \in \p[I] : \Delta_{S,T}(\lambda) = 0\text{ whenever $I = S\sqcup T$ with $S$, $T$ both nonempty}\}\]
for nonempty sets $I$.
It follows by general results of Stover \cite{stover} that every connected Hopf monoid which is commutative and cocommutative is isomorphic to the {free commutative Hopf monoid} on its subspecies of primitive elements; see Theorem \ref{coco-thm}.
Here, the \emph{free commutative Hopf monoid} on a vector species is an analogue of the symmetric algebra on a graded vector space; we review the precise definition in Section \ref{free-sect}.

 This discussion suggests that we  try to classify which free commutative  Hopf  monoids are freely self-dual. This may be a difficult problem.
For such a Hopf monoid to be FSD it must possess a basis, but examples show that, at least for arbitrary fields $\kk\subset \CC$, this necessary condition is not sufficient; see the remarks following Theorem \ref{coco-cor}.
To obtain a more tractable classification, we might restrict our attention to  connected Hopf monoids which are FSD and which have some additional property.
For example, one might consider the following:

\begin{problem*}\label{psd-problem}
Classify the connected Hopf monoids $\h = (\p,\nabla,\Delta)$ which are FSD with respect to some basis  and  whose structure constants $a_{\lambda',\lambda''}^{\lambda} = b_{\lambda',\lambda''}^{\lambda}$ in this basis are all nonnegative integers.
\end{problem*}

Connected Hopf monoids with these properties would be the obvious species analogues of Zelevinksy's notion of \emph{positive self-dual} (or \emph{PSD} for short) graded connected Hopf algebras. 
Zelevinky's monograph \cite{zel} (see also \cite[Section 3]{ReinerNotes}) develops a detailed structure theory for such Hopf algebras: each PSD Hopf algebra has a unique factorization as a tensor product of Hopf algebras isomorphic (after rescaling  gradings) to the familiar algebra of symmetric functions.

One hopes to find a similarly detailed structure theory  in answer to our problem above. Towards this end, 
we consider  the following special case:
namely, we  classify those  connected Hopf monoids $\h =(\p,\nabla,\Delta)$ such that
\begin{itemize}
\item[(a)] $\h$ is freely self-dual with respect to some basis $\P$ for $\p$.
\item[(b)] The product $\nabla$ is linearized with respect to the same basis $\P$.
\end{itemize}
Here, following \cite[Section 8.7.2]{species}, by \emph{linearized} we mean that each  component $\nabla_{S,T}$ of the product
is the linearization of a map $\P[S]\times \P[T]\to \P[S\sqcup T]$.
We refer to the  class of such connected Hopf monoids 
as \emph{strongly self-dual} (or \emph{SSD} for short).
The structure constants of a strongly self-dual Hopf monoid  are all nonnegative integers; in particular, they are all zero or one.
Many natural examples of FSD Hopf monoids   have this  property, including $\bfE_C$ and $\bfPi$.


Our results on SSD Hopf monoids appear in Section \ref{mufsd-sect}.
Proposition \ref{redef-prop} gives an elementary alternate definition of strongly self-dual and 
Theorem \ref{selfcompat-thm} characterizes the linearized products which give rise to  SSD Hopf monoids. From the description of an FSD Hopf monoid as a free commutative Hopf monoid
in Section \ref{structure-sect}, we derive this theorem:

\begin{letterthm}{B}\label{B}(See Theorem \ref{fsd-thm}.) A connected Hopf monoid which is commutative and cocommutative is strongly self-dual if and only if its subspecies of primitive elements has a finite basis.
\end{letterthm}

Let $\Q$ be a set species.
A \emph{$\Q$-labeled set partition} $X$ is then
a
  set of pairs $(B,\lambda)$, which we call the \emph{blocks} of $X$, consisting of a nonempty finite set  $B$ and an element  $\lambda \in \Q[B]$, subject to the condition that 
$B \cap B' = \varnothing$
whenever 
$(B,\lambda)$ and $(B',\lambda')$ are distinct
blocks. 
We say that $X$ is a partition of the set which is the union of its unlabeled blocks $B$. 

The collection of $\Q$-labeled set partitions forms a set species, which we denote $\cS(\Q)$, in the following way: for each finite set $I$ define 
$\cS(\Q)[I]$
as the set of $\Q$-labeled partitions of $I$,
%
%
and
for each bijection $\sigma : I \to I'$ between finite sets 
define $\cS(\Q)[\sigma]$
as the map sending $X \in \cS(\Q)[I]$ to the labeled partition $X' \in \cS(\Q)[I']$ such that 
\[(B,\lambda)\in X
\qquad\text{if and only if}
\qquad 
 (\sigma(B), \Q[\sigma](\lambda))\in X'.\]
This notation leads to the following concrete description of a strongly self-dual Hopf monoid.

\begin{letterthm}{C}\label{C}(See Theorem \ref{muFSD-mainthm}.)
If $\h$ is a connected Hopf monoid which is commutative and cocommutative, and if $\Q$ is a basis for its subspecies of primitive elements, then $\h$ is strongly self-dual with respect to a basis isomorphic to $\cS(\Q)$, 
and the action of the product on this basis may be identified with the disjoint union of $\Q$-labeled set partitions.
\end{letterthm}

In Section \ref{pifsd-sect} we consider the ``dual'' problem of classifying the connected Hopf monoids $\h =(\p,\nabla,\Delta)$ exhibiting the following variations of properties
(a) and (b) above:
 \begin{itemize}
 \item[(a$'$)] $\h$ is freely self-dual respect to some basis $\P$ for $\p$.
\item[(b$'$)] The coproduct $\Delta$ is linearized with respect to the same basis $\P$.

\end{itemize}
In this context, \emph{linearized} means that each  component $\Delta_{S,T}$ of the coproduct
is the linearization of a map $\P[S \sqcup T] \to \P[S]\times \P[ T]$.
We refer to the  class of  connected Hopf monoids with  these two properties
as \emph{linearly self-dual} for reasons which will become clear with the next theorem.
The structure constants of such Hopf monoids are again all  zero or one.

Despite the symmetry of our definitions,  linear self-duality implies strong self-duality.
We have the following pair of results.
Here, we write $\p[1]$ for $\p[\{1\}]$, and more generally we  write  $\p[n]$ for $\p[\{1,\dots,n\}]$.
Additionally, we say that a connected Hopf monoid $(\p,\nabla,\Delta)$ is \emph{linearized} with respect to some basis $\P$ if both its product and coproduct are linearized in this basis.

\begin{letterthm}{D}\label{D}(See Corollary \ref{deltanabla-cor} and Theorem \ref{pi-mainthm}.)
If $\h$ is a linearly self-dual connected Hopf monoid and $C$ is any basis for $\p[1]$, then  
$\h \cong \bfE_C$. In particular, a connected Hopf monoid which is linearly self-dual in some basis  is  strongly self-dual (and hence linearized) in the same basis.
\end{letterthm}


\begin{letterthm}{E}\label{E}(See Theorem \ref{pi-lastthm}.)
A connected Hopf monoid $\h$ which is commutative and cocommutative is linearly self-dual if and only if its subspecies of primitive elements $\cP(\h)$ is such that $\cP(\h)[1]$ is finite-dimensional and  $\cP(\h)[n]= 0$ whenever $n\neq 1$. 
\end{letterthm}


In Section \ref{linear-sect} we turn our attention to linearized  Hopf monoids.
Our main results concern the case when such Hopf monoids are strongly self-dual; namely:
\begin{letterthm}{F}\label{F}(See Theorem \ref{coco-cor} and Corollary \ref{last-cor}.)
Let $\h$ be a connected Hopf monoid which is linearized in some basis. Then $\h$ is strongly self-dual (in some other basis) if and only if $\h$ is finite-dimensional, commutative, and cocommutative. 
\end{letterthm}

For Hopf monoids which are linearized and strongly self-dual, we prove that the product and coproduct determine a ``partial order'' on the Hopf monoid's basis  (generalizing the refinement order on set partitions), and  that conversely the product and this partial order uniquely determine the coproduct.
%
%
%
 In greater detail, if $\P$ is a set species then let $\P\times \P$ denote the set species with 
 \[(\P\times \P)[I] = \P[I]\times \P[I] \qquand (\P\times \P)[\sigma] = \P[\sigma]\times \P[\sigma]\] for finite sets $I$ and bijections $\sigma$. We say that a subspecies $\cO \subset \P\times \P$ is \emph{transitive} if $(a,b),(b,c) \in \cO[I]$ implies $(a,c) \in \cO[I]$. A \emph{partial order} on $\P$ is then a transitive subspecies $\cO\subset \P\times \P$ for which $(a,b) \in \cO[I]$ implies $a\neq b$.
Given these definitions, we have the following theorem.

\begin{letterthm}{G}\label{G}(See Theorem \ref{order-thm}.) 
Let $\h= (\p,\nabla,\Delta)$ be an SSD connected Hopf monoid which is linearized  in some basis $\P$.
 The (unique) minimal transitive subspecies $\cO\subset\P\times \P$ such that $(\lambda,\lambda') \in \cO[I]$   whenever $\lambda,\lambda' \in \P[I]$ are distinct elements and  $\lambda = \nabla_{S,T}\circ \Delta_{S,T} (\lambda')$ for some disjoint decomposition $I = S\sqcup T$
is then a partial order on $\P$.
\end{letterthm}

Proposition \ref{lattice-cor} shows that the
 lower intervals  
 of this order are finite lattices.
Lemma \ref{order-lem1} isolates two other key properties of the partial order $\cO$: namely,  (A)   lower intervals are preserved by products and (B) if  $I = S\sqcup T$ is a disjoint decomposition then each $\lambda \in \P[I]$ has a unique greatest lower bound in the image of $\nabla_{S,T}$. See Theorem \ref{order-thm2} for a more precise description of these properties.
Combining the remaining results of Sections \ref{order-sect} and \ref{bases-sect} 
gives this theorem:

\begin{letterthm}{H}\label{H}(See Theorems \ref{order-thm2} and \ref{basis-thm} and Corollary \ref{last-cor}.)
Let $\h=(\p,\nabla,\Delta)$ be a connected Hopf monoid which is strongly self-dual in some basis $\P$. 
Then $\h$ is linearized  (in some other basis) if and only if  there exists a partial order on $\P$ with  properties (A) and (B).
\end{letterthm}

%
%

While the scope of this paper is limited to  Hopf monoids in species, 
in the companion work \cite{supp} we apply  the results here to unify several Hopf algebra constructions appearing in the literature. A sketch of these applications goes as follows.
From any comonoid which is connected, cocommutative, and linearized, there is standard way of constructing a connected Hopf monoid which is strongly self-dual and linearized.
For Hopf monoids arising in this way, we can define four natural bases in terms of its associated partial order.
Applying the Fock functors $\cK$ and $\overline \cK$ to these Hopf monoids gives rise to graded connected Hopf algebras with four bases, which one can view as analogues of the four classical bases of the Hopf algebra of symmetric functions.
Of particular interest, 
we can obtain, via this construction,  the  Hopf algebras of superclass functions on the towers of maximal unipotent subgroups of the finite Chevalley groups of type A and D, studied in
\cite{AZ} and \cite{CB}.
Among other consequences, this realization of these Hopf algebras affords an easy proof that they are isomorphic whenever they have the same graded dimension.

\subsection*{Acknowledgements}

I thank Marcelo Aguiar for answering several questions about species and clarifying some of the background material in Section \ref{free-sect}. I am grateful also to Carolina Benedetti for many helpful discussions and for providing me with a preliminary version of  part of her PhD thesis.

\section{Preliminaries} 
\label{prelim-sect}

The primary reference for the  material in this preliminary section is Aguiar and Mahajan's book \cite{species}, and also the related papers 
\cite{species1,species2}. Other useful general references for species include  \cite{blp}; see also \cite{Mendez}.
The concept of a vector species originates in Joyal's work \cite{Joyal}, but can also be viewed a special case of the more general notion of an {FI-module} recently introduced in \cite{FI}. (While a vector species is a functor $\FB \to \Vec$, an \emph{FI-module} can be interpreted as a functor $\FI \to \Vec$ where $\FI$ is the category of finite sets with injective maps as morphisms.)

We will frequently refer in examples to the set species 
\[\E,\qquad \E_C,\qquad \fk S,\qquad  \Pi,\qquad \L.\]
Here $\E$, $\E_C$, $\Pi$, and $\L$ are the natural bases of $\bfE$, $\bfE_C$, $\bfPi$, and $\bfL$ in the introduction. Thus if $S$ is a finite set then $\E[S] = \{1_\kk\}$, while $\E_C[S]$, $\Pi[S]$, and $\L[S]$ are the respective sets of maps $S \to C$, partitions of $S$, and linear orders of $S$.
We define $\fk S$ as the species of permutations, so $\fk S[S]$ is the set of bijections  $S \to S$, and $\fk S[\sigma]$ for a bijection $\sigma : S \to S'$ is the map $\fk S[\sigma] : \lambda \mapsto \sigma \circ \lambda \circ \sigma^{-1}$.

\subsection{Connected species and the Cauchy product}\label{cauchy-sect}

The definition of a Hopf monoid sketched in the introduction derives from the following bifunctor on vector species, which Aguiar and Mahajan's book \cite{species} refers to as the \emph{Cauchy product}.

\begin{definition}\label{cauchy-def}
The \emph{Cauchy product} of two vector species $\p$ and $\q$ is the vector species  $\p\cdot\q$ 
such that 
\[ (\p\cdot\q)[I] = \bigoplus_{I=S\sqcup T} \p[S]\otimes \q[T]\qquand (\p\cdot\q)[\sigma] = \bigoplus_{I=S\sqcup T} \p[\sigma|_S]\otimes \q[\sigma|_T]\]
for finite sets $I$ and bijections $\sigma: I \to I'$ between finite sets, where the direct sums are over all ordered decompositions of $I$ into disjoint subsets $S$, $T$.
In turn, the \emph{Cauchy product} of two morphisms $f : \p \to \p'$ and $g:\q \to \q'$ is the morphism  $f\cdot g : \p\cdot \q \to \p'\cdot \q'$ whose $I$-component is
\[(f\cdot g)_I=\bigoplus_{I=S\sqcup T} f_S\otimes g_T\] where the direct sum is again over all ordered disjoint decompositions $I = S \sqcup T$.
\end{definition}

Many results concerning monoidal structures on species are greatly simplified when we restrict our attention to species $\p$ with a specified isomorphism $\p[\varnothing] \cong \kk$.
Following \cite[Section 8.9.1]{species}, we call such species \emph{connected}.
 The technical definition goes as follows; recall here the definition of the vector species $\one$ from \eqref{one-def}.


\begin{definition} 
A \emph{connected species} is a vector species $\p$ with morphisms
$\eta : \one \to \p$ and $\varepsilon : \p \to \one$
such that $\eta_\varnothing : \kk\to \p[\varnothing]$ and $\varepsilon_\varnothing : \p[\varnothing] \to \kk$ are inverse isomorphisms of $\kk$-vector spaces. 
We refer to $\eta$ and $\varepsilon$ as the \emph{unit} and \emph{counit} of 
$\p$. 
\end{definition}

Note that $\eta_S$ and $\varepsilon_S$ are automatically zero whenever $S$ is nonempty, and that $\eta$ and $\varepsilon$ determine each other.
If $\p$ is a vector species with $\p[\varnothing] = \kk$ (such as $\one$ or $\bfE$),
then,
unless explicitly noted otherwise,
we view $\p$ as the connected species whose unit and counit are the unique morphisms 
whose $\varnothing$-components are the identity maps on $\kk$.

Connected species form a category, which we denote $\Sp^\circ$, whose morphisms are natural transformations $f: \p \to \p'$ which preserve units and counits, in the sense that 
$f\circ \eta = \eta'$  and $\varepsilon =  \varepsilon'\circ f$, with $\eta'$ and $\varepsilon'$ denoting the unit and counit of $\p'$. 
Whether or not a morphism of vector species preserves  units and counits in this sense  reduces to the following simpler condition.

\begin{fact}\label{connected-morph-fact}
Let $f : \p \to \p'$ be a morphism of vector species. If $\p$ and $\p'$ are connected with units $\eta$ and $\eta'$, 
then $f$ is a morphism of connected species if and only if 
$f_\varnothing $ maps $\eta_\varnothing(1_\kk) \mapsto \eta'_\varnothing(1_\kk)$.
\end{fact}

\begin{proof}
The given condition holds if $f\circ \eta = \eta'$ and, by linearity, only if  $\varepsilon =  \varepsilon'\circ f$. Now observe that since $\eta_\varnothing$ and $\epsilon_\varnothing$ are inverses, $f\circ \eta = \eta'$ holds if and only if $\varepsilon =  \varepsilon'\circ f$.
\end{proof}

Let $\P$ be a set species and recall the definition of its linearization $\kk\P$ from the introduction. If $\P[\varnothing] $ consists of exactly one element (which we will denote $1_\P$), then we view the  $\kk\P$
as a connected species  with respect to the unit $\eta : \one \to \kk\P$ and counit $\varepsilon : \kk\P \to \one$ whose $\varnothing$-components are the  $\kk$-linear maps with 
$\eta_\varnothing : 1_\kk \mapsto 1_\P$ and $ \varepsilon_\varnothing : 1_\P \mapsto 1_\kk$.
If $\p$ and $\p'$ are connected species, then the Cauchy product $\p \cdot \p'$ is also connected; its unit and counit are the morphisms
$\one \to \p\cdot \p'$ and $ \p\cdot \p'\to \one$ whose $\varnothing$-components are 
 the respective compositions
\[\label{cauchy-conn} \kk \xrightarrow{\sim} \kk\otimes \kk \xrightarrow{\eta_\varnothing\otimes\eta_\varnothing'} \p[\varnothing]\otimes\p'[\varnothing]
\qquand
\p[\varnothing]\otimes\p'[\varnothing]\xrightarrow{\varepsilon_\varnothing \otimes \varepsilon_\varnothing'} \kk\otimes\kk \xrightarrow{\sim}\kk\]
where as usual $\eta$ and $\varepsilon$ (respectively $\eta'$ and $\varepsilon'$) denote 
the unit and counit of $\p$ (respectively $\p'$).
With respect to these conventions, the Cauchy product of two morphisms of connected species is again a morphism of connected species (i.e., it preserves the unit and counit).
The Cauchy product, in particular, defines a bifunctor on connected species,
and makes this category symmetric monoidal (with unit object the connected species $\one$); see \cite[Section 8.1.2]{species}.

\subsection{Monoidal structures on connected species} 
\label{monoidal-struct-sect}

In any symmetric monoidal category there are standard notions of (commutative) monoids, (cocommutative) comonoids, bimonoids, and Hopf monoids.  We review explicitly the details of these constructions here, in the particular category of connected species with the symmetric monoidal structure given by the Cauchy product. 
Our definitions are taken from \cite[Sections 8.2 and 8.3]{species}, where they appear in slightly greater generality for vector species which are not necessarily connected. These details supplement the  informal definitions in the introduction.

To begin, we introduce the following notation for morphisms to and from Cauchy products. Throughout,  $\p$, $\q$, and $\r$ are vector species. Given morphisms  $f : \p \cdot \q \to \r$ and $g:\p \to \q \cdot \r$ and disjoint finite sets $S$, $T$, we let 
\[ f_{S,T}:  \p[S]\otimes \q[T] \to \r[S\sqcup T]
\qquand
 g_{S,T} : \p[S\sqcup T ] \to \q[S]\otimes \r[T]
\] denote the $\kk$-linear  maps 
 given by respectively pre- and post-composing the $(S\sqcup T)$-components of $f$ and $g$ 
 with the canonical inclusion and  projection 
 \[  \p[S]\otimes \q[T] \to (\p\cdot \q)[S\sqcup T]
\qquand
(\q\cdot \r)[S\sqcup T] \to \q[S]\otimes \r[T]
.\] 
Of course, it suffices to specify the system of maps $(f_{S,T})$ and $(g_{S,T})$ to uniquely determine the morphisms $f$ and $g$.


The following  definitions clarify the notions of associativity and commutativity for morphisms to and from Cauchy products.
 In the commutative diagrams in these definitions and elsewhere, whenever we write 
\[
\begin{diagram}
V_1 \otimes V_2 \otimes \dots \otimes V_k && \rTo^{\sim} && V_{\sigma(i)} \otimes V_{\sigma(2)} \otimes \dots \otimes V_{\sigma(k)}
  \end{diagram}
  \]
  where each $V_i$ is a $\kk$-vector
  space and $\sigma \in S_k$, we mean the obvious isomorphism given by linearly extending the map on tensors 
$v_1 \otimes v_2 \otimes \dots \otimes v_k  \mapsto v_{\sigma(i)} \otimes v_{\sigma(2)} \otimes \dots \otimes v_{\sigma(k)}$.

\begin{definition}\label{assoc-def}
A morphism $\nabla : \p \cdot \p \to \p$ is \emph{associative} (respectively, \emph{commutative}) if the diagram on the left (respectively, right)
\[
\begin{diagram}
  \p[R]\otimes \p[S]\otimes \p[T] &&& \rTo^{\nabla_{R, S} \otimes \id} &&& \p[R\sqcup S]\otimes \p[T] \\
\dTo^{\id\otimes \nabla_{S,T}}&  &&&&& \dTo_{\nabla_{R\sqcup S,T}} \\
\p[R] \otimes \p[S\sqcup T] & &&\rTo^{\nabla_{R,S\sqcup T}}  &&& \p[ R\sqcup S \sqcup T] \end{diagram}
\qquad
\begin{diagram}
  \p[S]\otimes \p[T] && \rTo^{\sim} && \p[T]\otimes \p[S] \\
& \rdTo_{\nabla_{S,T}} & &\ldTo_{\nabla_{T,S}} \\
& & \p[S\sqcup T] 
\end{diagram}
 \]
 commutes for all pairwise disjoint finite sets $R$, $S$, $T$.
 \end{definition}

\begin{definition}\label{coassoc-def}
A morphism $\Delta : \p  \to \p \cdot \p$ is \emph{coassociative} (respectively, \emph{cocommutative}) if the diagram on the left (respectively, right)
 \[
\begin{diagram}
  \p[R]\otimes \p[S]\otimes \p[T] &&& \lTo^{\Delta_{R, S} \otimes \id} &&& \p[R\sqcup S]\otimes \p[T] \\
\uTo^{\id\otimes \Delta_{S,T}}&  &&&&& \uTo_{\Delta_{R\sqcup S,T}} \\
\p[R] \otimes \p[S\sqcup T] & &&\lTo^{\Delta_{R,S\sqcup T}}  &&& \p[ R\sqcup S \sqcup T] \end{diagram}
\qquad
\begin{diagram}
  \p[S]\otimes \p[T] && \rTo^{\sim} && \p[T]\otimes \p[S] \\
& \luTo_{\Delta_{S,T}} & &\ruTo_{\Delta_{T,S}} \\
& & \p[S\sqcup T] 
\end{diagram}
 \]
 commutes for all pairwise disjoint finite sets $R$, $S$, $T$. 
 
 \end{definition}

Let $\nabla : \p \cdot \p \to \p$ and  $\Delta : \p\to\p \cdot \p$ be morphisms of vector species and suppose we have a decomposition  $I = S_1 \sqcup S_2 \sqcup \cdots \sqcup S_k$ into pairwise disjoint finite sets.  There are in general $(k-1)!$ distinct $\kk$-linear maps 
\be\label{nablam}   \p[S_1] \otimes \p[S_2] \otimes \cdots \p[S_k] \to \p[I]\ee
which can be formed by composing maps of the form $\id \otimes \cdots  \otimes \nabla_{X,Y}\otimes \cdots \otimes \id$. When $k=1$ and $k=2$, the unique maps of this type   are  the identity  on $\p[I]$ and $\nabla_{S_1,S_2}$.
Likewise, there are in general
$(k-1)!$ distinct $\kk$-linear maps 
\be\label{deltam}   \p[I] \to \p[S_1] \otimes \p[S_2] \otimes \cdots \p[S_k] \ee
which can be formed by composing maps of the form $\id \otimes \cdots  \otimes \Delta_{X,Y}\otimes \cdots \otimes \id$. 
When $\nabla$
is
associative,
the compositions defining 
\eqref{nablam} all coincide, and 
when $\Delta$ is coassociative, the compositions defining \eqref{deltam} all coincide.
In these respective cases, we denote the resulting maps \eqref{nablam} and \eqref{deltam} 
by
\[
\nabla_{S_1,S_2,\dots,S_k}
\qquand
\Delta_{S_1,S_2,\dots,S_k}.\]
When $k=1$   both of these maps are just the identity on $\p[I]$.

We now define the notion of Hopf compatibility of morphisms
mentioned in the introduction.

\begin{definition}\label{hopf-compat-def} 
Two morphisms $\nabla : \p \cdot \p \to \p$ and $\Delta : \p \to \p\cdot \p$ are \emph{Hopf compatible} if 
for any
two disjoint decompositions $I = R\sqcup R' = S\sqcup S'$ of the same finite set,
   the diagram
\[
\begin{diagram}
\p[R] \otimes \p[R'] &\rTo^{\nabla_{R,R'}}& & \p[I] &&\rTo^{\Delta_{S,S'}} & \p[S]\otimes\p[S'] \\  
\dTo^{\Delta_{A,B}\otimes \Delta_{A',B'}}&&  &&  && \uTo_{\nabla_{A,A'}\otimes \nabla_{B,B'}} \\  
\p[A] \otimes \p[B] \otimes \p[A'] \otimes \p[B']&&& \rTo^{\sim} &&&\p[A] \otimes \p[A'] \otimes \p[B] \otimes \p[B'] 
 \end{diagram}
 \]
 commutes, where 
$A = R \cap S 
$ and
$B = R\cap S'
$ and
$A' = R'\cap S
$ and
$B' = R' \cap S'.
$
\end{definition}

Using these definitions, we may  succinctly describe the notions of monoids, comonoids, and Hopf monoids  in the symmetric monoidal category of connected species (with the Cauchy product). We refer to these particular constructions as \emph{connected monoids}, \emph{connected comonoids}, and \emph{connected Hopf monoids}.
\begin{itemize}

\item A \emph{connected monoid} is a pair $(\p,\nabla)$ where $\p$ is a connected species (with unit $\eta$) and $\nabla : \p\cdot \p \to \p$  is an associative morphism (called the product)
which is {unit} compatible in the sense that if $1_\p = \eta_\varnothing(1_\kk)$ then
\[ \nabla_{\varnothing,S}(1_\p \otimes \lambda) = \nabla_{S,\varnothing}(\lambda \otimes 1_\p) =   \lambda\]
for all $\lambda \in \p[S]$ and
 finite sets $S$. 
 A morphism  $ (\p,\nabla) \to (\p',\nabla')$ between  connected monoids is a morphism   $f : \p \to \p'$ of connected species  which commutes with the product morphisms in the sense that 
$f_{S\sqcup T}\circ \nabla_{S,T} = \nabla_{S,T}' \circ (f_{S} \otimes f_{T})$ for all  disjoint finite sets $S$, $T$.  
 A connected monoid if \emph{commutative} is its product  is commutative.

\item A \emph{connected comonoid} is a pair $(\p,\Delta)$ where $\p$ is a connected species (with counit $\varepsilon$) and $\Delta :\p\to \p\cdot \p $ is a coassociative morphism (called the coproduct)
which is {counit} compatible in the sense that if $1_\p = \eta_\varnothing(1_\kk)$ then
\[ \Delta_{\varnothing,S}(\lambda) = 1_\p \otimes \lambda \qquand \Delta_{S,\varnothing}(\lambda) =   \lambda\otimes 1_\p\]
for all $\lambda \in \p[S]$ and finite sets $S$.  
 A morphism  $ (\p,\Delta) \to (\p',\Delta')$ between  connected comonoids is a morphism  $f : \p \to \p'$  of connected species which commutes with the coproduct morphisms in the sense that 
$(f_S\otimes f_T)\circ \Delta_{S,T} = \Delta_{S,T}' \circ f_{S\sqcup T}$ for all  disjoint finite sets $S$, $T$. 
A connected comonoid if \emph{cocommutative} is its coproduct  is cocommutative.

\item A \emph{connected Hopf monoid} is a triple $(\p,\nabla,\Delta$) where $\p$ is a connected species and  $\nabla : \p\cdot \p \to \p$ and $\Delta : \p \to \p\cdot \p$ are  Hopf compatible morphisms such that $(\p,\nabla)$ is a connected monoid and $(\p,\Delta)$ is a connected comonoid. A morphism of connected Hopf monoids is a morphism of connected species which commutes with products and coproducts. 
A connected Hopf monoid is \emph{commutative} or \emph{cocommutative} if it is commutative as a monoid or cocommutative as a comonoid.
 \end{itemize}
 
 Concerning these definitions we make two remarks. First, 
\emph{a prior} we only  require that $\nabla$ and $\Delta$ are  morphisms in the category of vector species, since the unit and counit compatibility conditions automatically imply that these morphisms are morphisms of connected species.
Second, we observe
that
the definition of a connected Hopf monoid given here
 is also that of a bimonoid in the symmetric monoidal category of connected species,  since in this category the notions of Hopf monoids and bimonoids coincide (see \cite[Section 8.4.1]{species}).
 In general, a Hopf monoid in a symmetric monoidal category is a particular type of bimonoid$-$namely, one such that the identity morphism in the corresponding convolution monoid has an inverse (referred to as the \emph{antipode}); see \cite[Section 1.2.5]{species} for the details of this general definition.

If $\h = (\p,\nabla,\Delta)$ is a connected Hopf monoid then its  antipode is the morphism $\S : \h \to\h$ given by the following formula  \cite[Proposition 8.13]{species}: the $\varnothing$-component $\S_\varnothing$ is the identity map on $\p[\varnothing]$, while the component $\S_I$ for nonempty sets $I$ is the map
\be\label{take} \S_I = \sum (-1)^k \cdot \nabla_{S_1,\dots,S_k} \circ \Delta_{S_1,\dots,S_k}\ee
where the sum is over all ordered decompositions of $I= S_1\sqcup \dots \sqcup S_k$ into nonempty subsets.
The formula \eqref{take}, which \cite{species} refers to as \emph{Takeuchi's antipode formula}, often involves numerous cancellations, and it is sometimes an interesting problem to find a simpler form of this alternating sum for particular Hopf monoids; see, e.g., \cite{BBT}.

Before ending this section we note one property of the product and coproduct of a connected Hopf monoid which will be of use later.

\begin{lemma}\label{injsurj-fact} Suppose $(\p,\nabla,\Delta)$ is a connected Hopf monoid. Then $\Delta_{S,T} \circ \nabla_{S,T}$ is the identity map on $\p[S]\otimes\p[T]$ for any disjoint sets $S$, $T$; in particular, $\nabla_{S,T}$ is  injective and $\Delta_{S,T}$ is  surjective.
\end{lemma}

\begin{proof}
Let $S$ and $S'$ be disjoint finite sets
and   take $R=S$ and $R'=S'$ in Definition \ref{hopf-compat-def}, so that $A = S$ and $B' = S'$ and $A'=B=\varnothing$. Write $1_\p$ for the image of $1_\kk \in \kk$ under the $\varnothing$-component of the unit of $\p$. 
If $h$, $g$, and $f$ denote the left, bottom middle, and right arrows in the diagram in Definition \ref{hopf-compat-def}, then $\Delta_{S,S'} \circ \nabla_{S,S'} = f\circ g \circ h$.
On the other hand,
the unit and counit compatibility of $\nabla$ and $\Delta$  implies that
if
$\lambda \in \p[S]$ and $\lambda' \in \p[S']$ then
\[ f\circ g \circ h(\lambda\otimes \lambda') = f\circ g(\lambda\otimes 1_\p \otimes 1_\p \otimes \lambda')
= f(  \lambda\otimes 1_\p \otimes 1_\p \otimes \lambda' ) = \lambda\otimes \lambda'.\] 
By linearity  $\Delta_{S,S'} \circ \nabla_{S,S'} = f\circ g \circ h $ is therefore the identity map
on $\p[S]\otimes \p[S']$.
\end{proof}

 \subsection{Duality for vector species}
 \label{dual-sect}
 
Here we briefly review the notion of duality for vector species and connected Hopf monoids, as detailed in \cite[Section 8.6]{species}.
Given a $\kk$-vector space $V$, let $V^*$ denote the $\kk$-vector space of $\kk$-linear maps $V \to \kk$, and given a $\kk$-linear map $f : V \to W$,  let  $f^* : W^* \to V^*$ denote the $\kk$-linear map with the formula $f^*(\varphi )= \varphi \circ f$.
One extends this notation to species in the following way.

 \begin{definition}
 The \emph{dual} of a vector species $\p$ is the vector species $\p^*$ with
\[
\p^*[S] = (\p[S])^*\qquand \p^*[\sigma] = (\p[\sigma^{-1}])^*
\]
 for finite sets $S$
 and bijections $\sigma$ between finite sets.
 In turn, the \emph{dual} of a morphism $f : \p \to \q$ of vector species is the natural transformation $f^* : \q^* \to \p^*$ with $S$-component $(f^*)_S = (f_S)^*$.
\end{definition} 

These definitions make $*$ into a contravariant functor from the category of vector species to itself.
The species $\one$ is canonically isomorphic to its dual $\one^*$ via the natural transformation whose $\varnothing$-component is the isomorphism $\KK \xrightarrow{\sim} \KK^*$ sending $c \in \KK$ to the unique linear map $\KK \to \KK$ with $1_\kk \mapsto c$.
This allows us to view the dual of a connected species as also connected, in the sense of this definition: 

\begin{definition} The \emph{dual} of a connected species $\p$  (with unit $\eta$ and counit $\varepsilon$) is the connected species given by $\p^*$ with respect to the unit  and counit
$ \one \xrightarrow {\sim} \one^* \xrightarrow{\varepsilon^*} \p^* $ and $
\p^* \xrightarrow{ \eta^*} \one^* \xrightarrow{\sim} \one. $
\end{definition}

With respect to this definition, the dual of a morphism $\p \to \q$ of connected species is automatically a morphism $\q^* \to \p^*$ of connected species.
The canonical vector space isomorphisms $(V\oplus W)^* \cong V^* \oplus W^*$ and $(V\otimes W)^* \cong V^* \otimes W^*$ lead to a canonical isomorphism of vector species $(\p\cdot \q)^* \cong \p^* \cdot \q^*.$
With respect to this identification, the following statements hold:
\begin{itemize}

\item If $\p$ is a connected monoid (with product $\nabla$) then the connected species $\p^*$ is a connected comonoid with respect to the coproduct 
$\p^* \xrightarrow{\nabla^*} (\p\cdot \p)^* \xrightarrow{\sim} \p^*\cdot \p^*.
$

\item If $\p$ is a connected comonoid (with coproduct $\Delta$) then the connected species $\p^*$ is a connected monoid with respect to the product 
$
\p^*\cdot \p^*   \xrightarrow{\sim}  (\p\cdot \p)^*  \xrightarrow{\Delta^*}  \p^* .
$

\item If $\p$ is a connected Hopf monoid (with product $\nabla$ and coproduct $\Delta$), then the connected species $\p^*$ is  a connected Hopf monoid with respect to the product  and coproduct just defined.
\end{itemize}
We define the \emph{dual} of a connected monoid, comonoid, or Hopf monoid  respectively as these connected comonoid, monoid, or Hopf monoid structures on $\p^*$. The dual of a commutative  monoid is a cocommutative comonoid and vice versa.
A connected Hopf monoid which is isomorphic to its dual is \emph{self-dual}. 

\subsection{Free commutative Hopf monoids}\label{free-sect}

If $\mathfrak q$ is a graded $\kk$-vector space with $\mathfrak q_0 = 0$, then the symmetric algebra $\mathrm{Sym}(\mathfrak q)$  has the structure of a graded Hopf algebra  which is connected, commutative, and cocommutative; see \cite[Example 1.18]{ReinerNotes}. Conversely,  the Cartier-Milnor-Moore theorem implies that every such graded Hopf algebra  arises in this way as the symmetric algebra generated by its subspace of primitive elements; see \cite[Remark 3.8.3]{Cartier}.

We review here the species analogue of these statements, which can be found in  \cite[Section 11.3]{species} and also \cite[Section 7.2]{species2}.
To begin, 
we say that a vector species $\q$ is \emph{positive} if $\q[\varnothing] = 0$, and write $\Sp_+$ to denote the full subcategory of positive vector species.
There is an adjoint equivalence 
between this category and the category of connected species $\Sp^\circ$ \cite[Proposition 8.43]{species}, and so working with positive rather than connected species is essentially a notational convention.
The species analogue of the symmetric algebra on a graded vector space is given by a functor 
\be\label{s2-def} \cS : \Sp_+ \to \cocoHopf(\Sp^\circ)\ee
where $\cocoHopf(\Sp^\circ)$ denotes the full subcategory of connected Hopf monoids which are commutative and cocommutative.
The definition of this functor from  \cite[Section 11.3.2]{species} goes as follows.

Most concisely, given a positive species $\q$ and a morphism $f : \q \to \q'$ of positive species,
we may define 
\[\cS(\q) = \bfE \circ \q
\qquand \cS(f) = \id_{\textbf{E}} \circ f\]
where $\circ$ denotes the \emph{substitution product} described in \cite[Definition 8.5]{species}.
More explicitly, 
$\cS(\q)$ is  the vector species such that if $I$ is a finite set then
\[ \cS(\q)[I] = \bigoplus_{X \in \Pi[I]} \q(X)
\qquad\text{where}
\qquad
 \q(X) \omdef= \bigotimes_{B \in X} \q[B].\]
The product on the right is the \emph{unordered tensor product}  over the blocks of the set partition $X$; see \cite[Example 1.30]{species} for more information on such unordered products.
In turn, if $\sigma : I \to I'$ is a bijection between finite sets then
$\cS(\q)[\sigma]$ is given by the direct sum over all partitions $X \in \Pi[I]$ of the unordered tensor product of maps $\bigotimes_{B \in X} \q[\sigma|_B]$.
%

If $X =\varnothing$ then we define $\q(X) = \kk$, so the species $\cS(\q)$ is automatically connected.
We view $\cS(\q)$ as a connected Hopf monoid with respect to the product and coproduct given by the morphisms  \[\nabla : \cS(\q)\cdot \cS(\q) \to \cS(\q) \qquand \Delta : \cS(\q) \to \cS(\q) \cdot \cS(\q)\] 
 whose components $\nabla_{S,T}$ and $\Delta_{S,T}$ for  disjoint finite sets $S$, $T$ are the  direct sums over all partitions $X\in \Pi[S]$ and $Y \in \Pi[ T]$ of the inverse  $\kk$-linear maps 
\[ \q(X) \otimes \q(Y) \xrightarrow{\sim} \q(X\sqcup Y)
\qquand
\q(X\sqcup Y)  \xrightarrow{\sim} \q(X) \otimes \q(Y).
\]
(Observe that  $\Delta_{S,T}$ thus acts as the zero map on $\q(Z)$ for all partitions $Z \in\Pi[ S\sqcup T]$ with at least one block which is not a subset of $S$ or of $T$.)
The triple $(\cS(\q),\nabla,\Delta)$ forms a connected Hopf monoid which is commutative and cocommutative,  which one calls the \emph{free commutative Hopf monoid} on $\q$.
We will usually write $\cS(\q)$ in place of the triple $(\cS(\q),\nabla,\Delta)$ when it is clear that we mean to indicate a connected Hopf monoid. However, we also write $\cS(\q)$ to refer to the underlying connected species of this Hopf monoid.


\begin{example}
Let $\bfX$ be the vector species with $\bfX[S] = \kk$ when $|S| = 1$ and $\bfX[S] = 0$ otherwise. Define $\bfX[\sigma]$ to be the identity map for any bijection $\sigma$ between finite sets.
One checks that $\cS(\bfX) \cong \bfE$ as defined in \eqref{one-def}.
If $\bfE_+$ denotes the maximal subspecies of $\bfE$ with $\bfE_+[\varnothing]=0$, then
one can show that $\cS(\bfE_+) \cong \bfPi$ as defined in Example \ref{Pi-ex}.
\end{example}

The antipode $\S$ of the connected Hopf monoid $\cS(\q)$ has the following simple formula: $\S_I$ acts on the subspace $\q(X) \subset \cS(\q)[I]$ for each set partition $X \in \Pi[I]$ as the scalar map $(-1)^k $, where $k$ is the number of blocks of $X$. This follows as a special case of \cite[Theorem 11.40]{species}.

\begin{remark}
The symmetric group $S_n$ acts linearly on the vector space $\cS(\q)[n]$ and makes this space a $\kk S_n$-module.
It appears to be a difficult problem to describe the irreducible decomposition of this module, or at least of the natural submodules \[M_\lambda \omdef= \bigoplus_{\lambda(X)=\lambda} \q(X)\] for integer partitions $\lambda$ of $n$. (Recall that an \emph{integer partition} of $n$ is a weakly decreasing sequence $\lambda = (\lambda_1,\lambda_2,\dots,\lambda_k)$ of positive integers which sum to $n$. Given a set partition $X$, we write $\lambda(X)$ for its \emph{type}, that is, the integer partition formed by the sizes of its blocks.)
In general, we may describe $M_\lambda$ as the following induced module.
Fix an integer partition $\lambda$ of $n$, and define $\q[\lambda]$ as the external tensor product $\q[\lambda_1]\otimes\q[\lambda_2]\cdots\otimes \q[\lambda_k]$. This is a module of the Young subgroup \[S_\lambda  = S_{\lambda_1}\times S_{\lambda_2}\times \dots\times S_{\lambda_k}.\] 
Write $N_\lambda$ for the normalizer of $S_\lambda$ in $S_n$. 
If $\lambda$ has $e_i$ parts of size $i$, then $N_\lambda$ is isomorphic to a semidirect product \[(S_{e_1}\times S_{e_2}\times \dots \times S_{e_n} ) \ltimes S_\lambda.\] This semidirect product structure allows us to view $\q[\lambda]$ to an $N_\lambda$-module in a standard way, and $M_\lambda$ is isomorphic to the $S_n$-module induced from this $N_\lambda$-module. 
%
When $\q = \bfE_+$, the extended modules ${\q[\lambda]}$ are all trivial representations, but even in this essentially simplest case of interest the decomposition of $M_\lambda$ is not known. It is a longstanding open problem to prove that $M_{(a,a,\dots,a)}$ is isomorphic to a submodule of $M_{(b,b,\dots,b)}$ when $n = ab$ and $a\leq b$, viewing both of these as submodules of $\cS(\bfE_+)[n]$. This statement is known as Foulkes' Conjecture; see \cite{Wildon2} for a recent overview.
%
%
%

\end{remark}

If $I$ is a nonempty finite set and $X$ is the unique partition of $I$ with one block, then $\q(X) = \q[I]$, and via this equality we  view 
 \be\label{inclu-eq}\q \subset \cS(\q)\ee
  as a subspecies.
We now arrive at the species analogue of the statement opening this section. This theorem follows as a special case of \cite[Theorem 120]{species2}, which is due originally to Stover \cite{stover}.

\begin{theorem}[Stover \cite{stover}] \label{coco-thm}
Let $\h = (\p,\nabla,\Delta)$ be a connected Hopf monoid and write $\q = \cP(\h)$ for its subspecies of primitive elements, as defined in the introduction. If 
$\h$ is commutative and cocommutative, then there exists a unique isomorphism
$\cS(\q) \xrightarrow{\sim} \h$ of connected Hopf monoids 
making the diagram 
\be\label{coco-diagram}
\begin{diagram}
\cS(\q) && \rTo^{\sim} && \p \\
& \luTo &&\ruTo\\
&& \q
 \end{diagram}
 \ee
commute, where the diagonal arrows are the natural inclusions of vector species.
\end{theorem}

Morphisms of connected Hopf monoids preserve primitive elements,
and we can therefore view $\cP$ as a functor 
$ \cP : \cocoHopf(\Sp^\circ) \to \Sp_+$, where $\cocoHopf(\Sp^\circ)$ denotes the full subcategory of connected Hopf monoids which are commutative and cocommutative,
by defining 
the  image under $\cP$ of a morphism of connected Hopf monoids to be the morphism between the corresponding species of primitive elements whose components are the obvious restrictions.
The functor   $\cS$ is left adjoint to this functor by \cite[Proposition 11.45]{species}. Thus  the connected Hopf monoid 
$\cS(\q)$ is \emph{free} on $\q$ with respect to the primitive element functor  $\cP$ and the inclusion \eqref{inclu-eq}. This explains  why we call $\cS(\q)$ a free commutative Hopf monoid. 
(The term \emph{free Hopf monoid} refers in \cite[Chapter 11]{species} to a related but distinct construction.)
One consequence of this fact is that the primitive elements of $\cS(\q)$ are just $\q$; i.e., $\cP(\cS(\q)) =\q$ \cite[Corollary 11.46]{species}.

\section{Freely self-dual  Hopf monoids}
\label{BSD-sect}

This section presents our results on connected Hopf monoids
which are 
 freely self-dual, strongly self-dual, and linearly self-dual.



\subsection{Some equivalent definitions}\label{fsd-sect}

Recall from the introduction that  a {basis} of a vector species $\p$ is a set species $\P$ such that $\p =\kk\P$. 
A \emph{basis} for 
  a connected species $\p$ (with unit $\eta$ and counit $\varepsilon$), is a set subspecies $\P \subset\p$ such that $\P[S]$ is  always a basis for the vector space $\p[S]$ and such that the maps $\eta_\varnothing$ and $\varepsilon_\varnothing$  restrict to bijections $\{1_\kk\} \to \P[\varnothing]$ and $\P[\varnothing] \to \{1 _\kk\}$.
The first condition means that $\p=\kk\P$ as vector species, while the second condition ensures that this equality holds in the sense of connected species.


%
  
Not every vector species has a basis. 
 We may characterize which do in the following way.
Observe that if $\p$ is a vector species then each vector space $\p[S]$ is naturally a module of the group of permutations of $S$; in particular, $\p[n]$ is a $\kk S_n$-module.
Whenever the symmetric group $S_n$ acts on a set $X$, the vector space 
$\kk X  = \kk\spanning\{x \in X\}$ generated by $X$ naturally carries the structure of a $\KK S_n$-module. We refer to such $\kk S_n$-modules
as \emph{permutation representations}.

\begin{proposition}\label{basis-prop}
A vector species $\p$ has a basis if and only if for each nonnegative integer $n$ the $\kk S_n$-module $\p[n]$ is  a permutation representation
of $S_n$.
\end{proposition} 

\begin{proof}
If $\p$ has a basis $\P$ then $\p[n]$ is clearly equal to the permutation representation arising from the action of $S_n$ on $\P[n]$. 
Conversely, if $\p[n]$ is isomorphic to a permutation representation for all $n$ then each  vector space $\p[n]$ has a basis $X_n$ which is permuted by $\p[\sigma]$ for all $\sigma \in S_n$. It is straightforward to check that if we define for each finite set $S$ of size $n$ 
\[ \P[S] = \p[\sigma](X_n)\qquad\text{where $\sigma : [n] \to X_n$ is any bijection}\]
then 
 $\P$ is a well-defined set subspecies of $\p$ which forms a basis.
\end{proof}

Let $\q$ be a positive species. If $\q$ has a basis $\Q$, then  $\cS(\q)$ has a basis which we may identify with the set species $\cS(\Q)$ of labeled partitions, as defined in the introduction.
On the other hand, one can construct a positive species $\q$ without a basis such that $\cS(\q)$ does have basis nevertheless (define $\q$ such that $\q[1]$ is the trivial representation, $\q[2]$ is the sign representation, and $\q[n]$ is the direct sum of many copies of the trivial representation when $n> 2$.) 
Of course, a free commutative Hopf monoid $\cS(\q)$ may fail to have a basis altogether, as the following example demonstrates.

\begin{example}\label{S-nobasis-ex}
Let $\q$ be the positive species given as follows.
Define $\q[S]$ to be zero when $|S| \notin \{1,2\}$ and to be $\kk$ otherwise.
Set $\q[\sigma]$ to be the identity map when $\sigma$ is a bijection between singleton sets, and to be the zero map when $\sigma$ is a bijection between sets with more than two elements.
 It remains only to specify the maps $\q[\sigma]$ when $\sigma$ is a bijection between two 2-element sets.
For this, first choose a bijection $\tau_S :  \{1,2\} \to S$ for each set $S$ with $|S| =2$. Then, given a bijection $\sigma: S \to S'$ between   2-element sets, define
\[ \q[\sigma] = \begin{cases} 1&\text{if }\tau_{S'}^{-1} \circ \sigma \circ \tau_S\text{ is the identity map on $\{1,2\}$} 
\\ -1 &\text{otherwise}.\end{cases}
\]
The vector space $\cS(\q)[n]$ has a basis indexed by the involutions (i.e., self-inverse elements) in the symmetric group $S_n$, and the representation of $S_n$ on this space is precisely the one studied in \cite{APR}. In particular,  by results of that paper,
 $\cS(\q)[n]$ is isomorphic as an $\kk S_n$-module to the multiplicity-free sum of all irreducible $\kk S_n$-modules. In general, this is not isomorphic to a permutation representation, since for $n$ sufficiently large not every irreducible representation of $S_n$ can appear in a multiplicity-free permutation representation, as a consequence of Wildon's classification of all such representations  \cite[Theorem 2]{Wildon}.
\end{example}

Let $\cF$ denote the forgetful functor from the category of vector species to set species.
Saying that a vector species $\p$  has a basis $\P$ is   equivalent to asserting that $\p$ is \emph{free} on $\P$ with respect to $\cF$. 
%
Consider the morphism $f: \P \to \cF(\p^*)$ whose components are the maps  $\P[I] \to \p^*[I]$ which send  $v \in \P[I]$ to the    linear functional 
$\delta_v : \p[I] \to \kk$  with $\delta_v(u) = \delta_{u,v}$ for all $u \in \P[I]$.
The defining universal property of a free object implies that $f$ has a unique extension to a morphism $f^{\mathrm{ext}} : \p \to \p^*$ making the  diagram 
\[
\begin{diagram}
\cF(\p) & & \rTo^{ \cF(f^{\mathrm{ext}}) } &&  \cF(\p^*) \\
 & \luTo & & \ruTo_{f} \\ & & \P
\end{diagram}
\]
commute, where the left diagonal arrow is the natural inclusion of set species. We refer $f^{\mathrm{ext}}$ as the 
%
 morphism $\p\to\p^*$ \emph{induced by  $\P$}. 
 This considerations motivate the following  definition. 

 \begin{definition} A connected Hopf monoid $\h = (\p,\nabla,\Delta)$ is \emph{freely self-dual} (or \emph{FSD} for short) with respect to a basis $\P$ for $\p$ if the   morphism $\p \to \p^*$ induced by $\P$
 defines an isomorphism $\h \cong \h^*$ of connected Hopf monoids. 
\end{definition}

%

Of course, a freely self-dual Hopf monoid is self-dual. Since the morphism $\p \to \p^*$ induced by a basis is always injective but is only surjective if $\p$ is finite-dimensional, all freely self-dual Hopf monoids are finite-dimensional.
The definition of freely self-dual here differs from the one in the introduction, but is equivalent by 
 the following proposition.
 
\begin{proposition}\label{structureconst-prop} A finite-dimensional connected Hopf monoid $(\p,\nabla,\Delta)$ is freely self-dual with respect to some basis $\P$ for $\p$ if and only if the structure constants of the product and coproduct in this basis coincide, in the sense that for any disjoint finite sets $S$, $T$ there 
are $c^\lambda_{\lambda',\lambda''} \in \kk$  
such that
\[ \nabla_{S,T}(\lambda'\otimes \lambda'') = \sum_{\lambda \in \P[S\sqcup T]} c^\lambda_{\lambda',\lambda''} \cdot \lambda
\qquand
\Delta_{S,T}(\lambda) = \sum_{(\lambda',\lambda'') \in \P[S]\times\P[T]} c^\lambda_{\lambda',\lambda''} \cdot \lambda'\otimes \lambda''\]
for all $\lambda \in \P[S\sqcup T]$ and $\lambda' \in \P[S]$ and $\lambda'' \in \P[T]$.
\end{proposition}

\begin{proof}
Assume $\p$ is a finite-dimensional connected species with a basis $\P$. The induced morphism $f : \p\to \p^*$ is then automatically an isomorphism.
Suppose $\h = (\p,\nabla,\Delta)$ is a connected Hopf monoid and let $\wt \h$ denote the image of the dual Hopf monoid $\h^*$ under $f^{-1}$.
Then $\p$ is the underlying species of both $\h$ and $\wt\h$, and $\h$ is FSD if and only if $\h = \wt \h$.
From this last statement, the lemma follows by noting that  the product (respectively, coproduct) of $\h$ has the same structure 
constants with respect to $\P$ as the coproduct (respectively, product) of $\wt \h$. 
\end{proof}

For the duration of this section assume $\KK\subset \CC$.
We can restate the preceding lemma in terms of the existence of a certain Hermitian form. 
By a \emph{sequilinear form} on a vector species $\p$, we
 mean a collection of sesquilinear maps $\langle \cdot ,\cdot \rangle : \p[S]\times {\p[S]}\to \CC$, one for each finite set $S$,
 such that
 \[\bigl\langle\  \p[\sigma](\lambda),\ \p[\sigma](\lambda')\ \bigr\rangle = \langle \lambda,\lambda'\rangle\qquad\text{for all $\lambda,\lambda' \in \p[S]$}\] whenever $\sigma :S\to S'$ is a bijection.
 If $\KK$ is closed under complex conjugation (i.e., $\KK = \overline \KK$), then  a sesquilinear form is equivalent
to a morphism \[\langle\cdot,\cdot\rangle: \p\times\overline{\p} \to \bfE\]
 where $\overline{\p}$ denotes the composition of $\p$ with the complex conjugation functor from $\Vec$ to itself $\times$ denotes the \emph{Hadamard product} of vector species; see  \cite[Definition 8.5]{species}.
The form $\langle\cdot,\cdot\rangle$ is \emph{Hermitian} if it always holds that $\langle\lambda,\lambda'\rangle =\overline{ \langle \lambda',\lambda\rangle}$
and \emph{positive definite} if  $\langle\lambda,\lambda\rangle > 0$ 
whenever $\lambda \neq 0$.

\begin{notation} If $\p$ and $\q$ are vector species then write $\p+\q$ for the vector species with 
\[ (\p+\q)[I] = \p[I]\oplus \q[I]\qquand (\p+\q)[\sigma] = \p[\sigma]\oplus \q[\sigma]\]
for finite sets $I$ and bijections $\sigma: I \to I'$ between finite sets.
\end{notation}

\begin{fact} Assume $\KK = \overline \KK$. If $\langle\cdot,\cdot\rangle$ is a positive definite Hermitian form on a vector species $\p$, and if $\q$ is a subspecies of $\p$, then there exists a unique subspecies $\r$ of $\p$ such that $\p =\q+\r$ and $\langle\lambda,\lambda'\rangle = 0$ for all   $\lambda \in \p[S]$ and $\lambda' \in \r[S]$ and finite sets $S$.
\end{fact}

\begin{proof}
Define $\r[S] = \{ y \in \p[S] : \langle x,y\rangle = 0\text{ for all }x \in \q[S]\}$ as the usual orthogonal complement of $\q[S]$ in $\p[S]$. Standard arguments from linear algebra show that $\p[S]=\q[S]\oplus\r[S]$, so it remains only to check that this definition  in fact makes $\r$  a subspecies of $\p$. This follows since if $\sigma : S\to S'$ is a bijection of finite sets and $y \in \r[S]$, then   $\langle x, \p[\sigma](y) \rangle = \langle \p[\sigma^{-1}](x),y\rangle = 0$ for all $x \in \q[S']$ since $\q$ is a subspecies,
so 
$\p[\sigma](y) \in \r[S']$.
\end{proof}  

We refer to the subspecies $\r$ as the \emph{orthogonal complement} of $\q$ in $\p$ with respect to 
 $\langle\cdot,\cdot\rangle$.

\begin{fact}
If a finite-dimensional vector species has a basis then it has a positive definite Hermitian form.
\end{fact}

\begin{proof}
Let $\p$ be a finite-dimensional vector species with a basis $\P$. The desired  form is given by the morphism $\p\times\overline\p\to\bfE$ whose $S$-component is the sesquilinear form $\p[S]\times\p[S]\to\kk$ with 
$x\otimes x \mapsto 1$ and $x\otimes y \mapsto 0$
for all $x,y \in \P[S]$ with $x\neq y$.
\end{proof}

\begin{notation}
We will denote the sesquilinear form   induced by a basis $\P\subset\p$ described in this proof of this fact by $\langle\cdot,\cdot\rangle_\P$.
\end{notation}

%

Given any sesquilinear form $\langle\cdot,\cdot\rangle$ on $\p$ and 
 finite sets $S$, $T$ and $x, y \in \p[S]\otimes  \p[T] $,
 we 
 write
$\langle x,y\rangle$ for the image of $x \otimes y$ under the unique linear map
$
 \p[S]\otimes   \p[T]  \otimes \p[S]\otimes   \p[T]  \to \kk
 $
  such that  
 \[
 \langle \lambda_1\otimes \lambda_1',\lambda_2\otimes \lambda'_2\rangle = \langle \lambda_1,\lambda_2\rangle \cdot \langle \lambda_1',\lambda'_2\rangle
\qquad\text{for any $\lambda_i\in \p[S]$ and $\lambda'_i \in \p[T]$.}\]
%
%
%
We now have our third characterization of freely self-dual Hopf monoids.

\begin{proposition} \label{form-prop}
A finite-dimensional connected Hopf monoid $(\p,\nabla,\Delta)$ is freely self-dual with respect to a basis $\P$ for $\p$ if and only if
 the  
 form
 $\langle\cdot,\cdot\rangle_\P$ is invariant in the sense
that
\[\Bigl \langle\ \nabla_{S,T}(x),\ y\ \Bigr\rangle_\P = \Bigl\langle\ x,\ \Delta_{S,T}(y)\ \Bigr\rangle_\P
\]
for all disjoint finite sets $S$, $T$  and  all $x \in \p[S]\otimes  \p[T]$ and $y \in \p[S\sqcup T]$.

\end{proposition}

\begin{proof}
Let $\h = (\p,\nabla,\Delta)$ be a finite-dimensional connected Hopf monoid and suppose $\P$ is a basis for $\p$. Write $a_{\lambda',\lambda''}^{\lambda}$ and $b_{\lambda',\lambda''}^{\lambda}$ for the respective structure constants of the product and coproduct of $\h$ with respect to $\P$.
If $\lambda' \in \P[S]$ and $\lambda'' \in \P[T]$ and $\lambda \in \P[S\sqcup T]$, then by  definition 
\[a_{\lambda',\lambda''}^{\lambda} =  \langle\ \nabla_{S,T}(x),\ y\ \rangle_\P 
\qquand
 b_{\lambda',\lambda''}^{\lambda} = \left\langle\ x,\ \Delta_{S,T}(y)\ \right\rangle_\P
 \] 
 when $x = \lambda'\otimes \lambda''$ and $y = \lambda$.
Since such elements $x$ and $y$ span $\p[S]\otimes \p[T]$ and $\p[S\sqcup T]$, it follows that the form $\langle\cdot,\cdot\rangle_\P$ is invariant if and only if the structure constants of the product and coproduct of $\h$ coincide  with respect to $\P$. The present proposition therefore follows from Proposition \ref{structureconst-prop}.
\end{proof}

\subsection{Structure of freely self-dual Hopf monoids}
\label{structure-sect}

%

Here we prove that connected Hopf monoids which are freely self-dual  are  both commutative and cocommutative. 
We fix some notation for the duration of this subsection. First, we let  
\[\h=(\p,\nabla,\Delta)\] 
be an arbitrary connected Hopf monoid. We will write $\one$ for the  subspecies of $\p$ with 
\be\label{one-newdef-eq}\one[\varnothing] = \p[\varnothing]
\qquand
\one[S]=0\text{ for all $S \neq \varnothing$.}\ee We view $\one$ as a connected species with the same unit and counit as $\p$. While this convention conflicts with our previous definition of $\one$,  if any confusion arises one can  simply assume without loss of generality that $\p[\varnothing] = \kk$. 
Next, we denote the species of  primitive elements of $\h$, as defined in the introduction, by
\be\label{q} \q = \cP(\h). \ee  Finally, we write 
$\r$ for the 
positive
subspecies of $\p$ with
 \be\label{r} \r[I] = \sum_{\substack{S\sqcup T = I \\ S,T \neq \varnothing}} \im(\nabla_{S,T})\ee for each finite set $I$, where the sum is over all ordered decompositions $S\sqcup T$ of $I$ into nonempty subsets.
Note that this sum has zero summands when  $|I| \leq 1$, and thus in these cases $\r[I]=0$.
Checking that $\r$ is indeed a subspecies,  using   the fact that $\nabla : \p \cdot \p \to \p$ is a morphism, is a straightforward exercise which we omit.

The following lemma uses the notion of the intersection of two subspecies, which we need to define. If $\x$ is a set or vector species and $\textbf{y}$, $\textbf{z}$ are both subspecies, then we denote by $\textbf{y}\cap\textbf{z}$ the subspecies of $\x$ with $(\textbf{y} \cap \textbf{z})[S] = \textbf{y}[S] \cap \textbf{z}[S]$
for all finite sets $S$.

\begin{lemma}\label{intersect-lem}
If $\q\cap \r = \bf{0}$ then $\nabla$ is commutative.
\end{lemma}

This lemma is analogous to \cite[Lemma AI.3]{zel} (which also appears as \cite[Lemma 3.5]{ReinerNotes}), 
and its proof is similar. 
 
\begin{proof}
We must show that whenever $S$, $T$ are disjoint finite sets and $\lambda \in \p[S]$ and $\lambda' \in \p[T]$, the element 
\be\label{xref} x\omdef =\nabla_{S,T}(\lambda\otimes\lambda') - \nabla_{T,S}(\lambda'\otimes\lambda) \in \p[S\sqcup T]\ee
 is zero.
 We proceed to do this by induction on $|S| + |T|$.
If $S$ or $T$ is empty then $x=0$ by the unit compatibility of $\nabla$. 
Assume 
therefore
that
$S$ 
and 
$T$ 
are 
both nonempty, and that elements of the form \eqref{xref} are zero whenever  $S$, $T$ are replaced by disjoint sets $S'$, $T'$ with $|S'|+|T'| < |S| +|T|$. 

Under these hypotheses $x$  belongs to $\r[S\sqcup T]$,
and so, since we assume $\q\cap \r = \bf{0}$, it suffices to show that also  $x \in \q[S\sqcup T]$.
To this end, let $R,R'$ be disjoint nonempty finite sets with $R\sqcup R' = S\sqcup T$; it is enough to check that  
\[\Delta_{R,R'}(x) = \Delta_{R,R'}\circ \nabla_{S,T}(\lambda\otimes\lambda') -  \Delta_{R,R'}\circ \nabla_{T,S}(\lambda'\otimes\lambda) = 0.\]
For this, let $A=R\cap S$ and $B =R\cap T$ and $A' = R'\cap S$ and $B' = R'\cap T$.
The Hopf compatibility of $\nabla$ and $\Delta$ then implies that $\Delta_{R,R'}(x)$ is a sum of elements of the form 
\be\label{4terms} \nabla_{A,B}(\alpha \otimes \beta) \otimes \nabla_{A',B'}(\alpha'\otimes \beta')
-
  \nabla_{B,A}( \beta \otimes \alpha) \otimes \nabla_{B',A'}(\beta' \otimes \alpha')
\ee
for some  $\alpha \in \p[A]$ and $\beta \in \p[B]$ and $\alpha' \in \p[A']$ and $\beta' \in \p[B']$.
We claim that each of these differences is zero. To see this, note that since $R$ and $R'$ are both nonempty, we have $|A|+|B| = |R| < |S|+|T|$ and $|A'|+|B'| =|R'| < |S| + |T|$.
Therefore \[\nabla_{A,B}(\alpha \otimes \beta)=\nabla_{B,A}( \beta \otimes \alpha) \qquand   \nabla_{A',B'}(\alpha'\otimes \beta') = \nabla_{B',A'}(\beta' \otimes \alpha')\] 
by induction, 
so 
the expressions \eqref{4terms} are all zero and $\Delta_{R,R'}(x) = 0$. We conclude that $x \in \q [S\sqcup T]$ which is what we needed to show that $x=0$.
\end{proof}

The next lemma shows that $\q \cap \r = \textbf{0}$ whenever $\h$ is freely self-dual.

\begin{lemma}\label{fsd-lem2}
Assume $\KK \subset \CC$.
If $\h$ is freely self-dual with respect to a basis $\P$ for $\p$, then the subspecies $\q$ and $\r$ are orthogonal with respect to the form $\langle\cdot,\cdot\rangle_\P$
and  we have
$\p =\one+ \q+\r$.
\end{lemma}

This lemma follows by a standard argument using the invariance of the form $\langle\cdot,\cdot\rangle_\P$; compare the proof of \cite[Proposition 3.3]{ReinerNotes} concerning Hopf algebras.

\begin{proof} 
It suffices to show that when
 $I$ is a nonempty finite set, the orthogonal complement of $\r[I]$ in $\p[I]$ is $\q[I]$.
  Proposition \ref{form-prop}
 shows that $\lambda \in \p[I]$ belongs to the orthogonal complement of $\r[I]$  if and only if whenever  $I = S\sqcup T$ is a disjoint decomposition into nonempty subsets, so we have  
\[\bigl\langle\ \lambda,\ \nabla_{S,T}(\alpha\otimes \beta)\ \bigr\rangle_\P=  \bigl\langle\ \Delta_{S,T}(\lambda),\ \alpha\otimes \beta\ \bigr\rangle_\P = 0\qquad\text{for all $\alpha \in \p[S]$ and $\beta \in \p[T]$.}\] 
The positive definiteness of $\langle\cdot,\cdot\rangle_\P$ implies that this holds if and only if $\Delta_{S,T}(\lambda) = 0$ whenever $S$ and $T$ are both nonempty, i.e., precisely when $\lambda \in \q[I]$.
%
\end{proof}

Combining the preceding lemmas show that an FSD connected Hopf monoid is commutative.
Invoking self-duality, we arrive at the following result, given as Theorem \descref{A} in the introduction.

\begin{theorem}\label{coco-cor}
Assume $\KK \subset \CC$.
An FSD connected Hopf monoid 
is commutative and cocommutative. 
Hence  if $\h$ is freely self-dual then there exists a unique isomorphism
$\cS(\q) \xrightarrow{\sim} \h$ of connected Hopf monoids 
making the diagram \eqref{coco-diagram} in Theorem \ref{coco-thm} commute.
\end{theorem}


\begin{remark}
Not every connected Hopf monoid of the form $\cS(\q)$ is also FSD, since its underlying species may not even possess a basis; see Example \ref{S-nobasis-ex}. 
Moreover, if the field $\kk$ is small enough, say $\kk = \QQ$, then one can construct a positive species $\q$ such that $\cS(\q)$ has a basis but is not FSD. We know of no such examples when $\kk $ is algebraically closed, and the following question is therefore open: are there any positive species $\q$ when $\kk=\CC$ such that $\cS(\q)$ has a basis but is not FSD?
\end{remark}


%
%
%
%

The previous theorem is all that we require for later sections, but before continuing we state a slightly more general result for which we can give a self-contained and constructive proof. 
From now on, 
we 
assume  that the (previously arbitrary) connected Hopf monoid $\h= (\p,\nabla,\Delta)$ is commutative.
Then, for each finite set $I$ and partition $X \in \Pi[ I]$, we may define a subspace
\[\nabla_X(\q) \subset \p[I]\]
in the following way. If $ I = \varnothing$ then set $\nabla_X(\q) = \p[\varnothing] \cong \kk$.
Otherwise, let $k\geq 1$ be the number of blocks of $X$ and choose a bijection $[k] \to X$. Write $S_i$ for the image of $i \in [k]$ under this map, and then set 
\be\label{nabla_X-eq} \nabla_X(\q) =\Bigl \{\ \nabla_{S_1,\dots,S_k}(x) : x \in \q[S_1]\otimes \cdots \otimes \q[S_k]\ \Bigr\} \subset \p[I].\ee
This is well-defined since the  set on the right has no dependence on the choice of ordering $[k]\to X$, as we assume that $\nabla$ is commutative.
Recall that $\nabla_{S_1,\dots,S_k}$ is the identity map when $k=1$, so if $X$ has a single block then $\nabla_X(\q) = \q[I]$.

The following statement implies Theorem \ref{coco-cor}  by Lemma \ref{fsd-lem2}, and gives a species analogue of the main theorem of \cite[Appendix A]{zel} which also appears as \cite[Theorem 3.7]{ReinerNotes}.
%
%

%

%

\begin{theorem} \label{contained-thm}
%
Suppose $\p=\one+\q+\r$, so that $\nabla$ is commutative. 
For any disjoint decomposition $I = S\sqcup T$ of a finite set, the following properties then hold: 

\ben
\item[(a)] There is a direct sum decomposition 
$\p[I] = \bigoplus_{X\in \Pi[I]} \nabla_X(\q)$.

\item[(b)] The maps $\nabla_{S,T}$ and $\Delta_{S,T}$ restrict  to inverse isomorphisms 
$\nabla_X(\q)\otimes \nabla_Y(\q)\cong\nabla_{X\sqcup Y}(\q)
$
for all set partitions  
 $X \in \Pi[S]$ and $Y \in \Pi[T]$. 

\item[(c)] The kernel of $\Delta_{S,T}$ 
 is  the direct sum of the subspaces $\nabla_Z(\q)$
over all partitions $Z \in \Pi[ I]$ with at least one block which is  not a subset of $S$ or of $T$.


\een
It follows from these properties that  $\Delta$ is  cocommutative, and that 
the  inclusion $\q \to \p$ extends uniquely to an isomorphism of connected Hopf monoids $\cS(\q) \xrightarrow{\sim} \h$. 

 \end{theorem}

\begin{proof}
Our proof of part (a) uses parts (b) and (c), so we prove these first.
For part (b), note by definition that  $\nabla_{S,T}$ restricts to a  surjective linear map $\nabla_X(\q) \otimes \nabla_Y(\q) \to \nabla_{X\sqcup Y}(\q)$. This map is  injective and hence bijective by Lemma \ref{injsurj-fact}, and by the same result its inverse is given by the restriction of $\Delta_{S,T}$.

Next, let $Z$ be a partition of $I$ with a block $I'$ contained in neither $S$ nor $T$. 
In light of part (b), to prove part (c) it suffices to show that $\nabla_Z(\q) \subset \ker \Delta_{S,T}$. 
To this end, 
define
\[
I'' = I-I'
\qquand
Z'  = \{I'\} \in \Pi[I']
\qquand
Z'' = Z - \{I'\} \in \Pi[I''].\]
Since $\nabla_{Z'}(\q) = \q[I']$, the subspace
$\nabla_Z(\q)$ is the image of  $\q[I']\otimes \nabla_{Z''}(\q)$ under $ \nabla_{I',I''}$ by part (b). We therefore need only show that 
\[\q[I']\otimes \nabla_{Z''}(\q)\subset \ker\(\Delta_{S,T}\circ \nabla_{I',I''}\).\]
For this, we note by the Hopf compatibility of $\nabla$ and $\Delta$ that if $A = I' \cap S$ and $B = I'\cap T$ and $A' = I''\cap S$ and $B' = I''\cap T$ then
 $\Delta_{S,T}\circ \nabla_{I',I''}$ is equal to the composition of three maps $f \circ \tau \circ g$, where 
$
g = \Delta_{A,B} \otimes \Delta_{A',B'}
$ and 
$f = \nabla_{A,A'}\otimes \nabla_{B,B'}$
 and  $\tau$ is the natural twisting isomorphism between the image of $g$ and the domain of $f$.
Since $I'$ is a subset of neither $S$ nor $T$, both $A$ and $B$ are nonempty sets, and so $\q[I'] \subset \ker \Delta_{A,B}$ by definition.
As such, it follows that
$\q[I']\otimes \nabla_{Z''}(\q)$ belongs to the kernel of $ g$, and hence also to the kernel of   $f\circ \tau \circ g = \Delta_{S,T}\circ \nabla_{I',I''}$, which is what we needed to show.

We now return to part (a).
First, we claim that 
if $X\neq Y$ are distinct partitions of $I$, then the subspaces $\nabla_X(\q)$ and $\nabla_Y(\q)$ have zero intersection.
This follows since if we let $S$ be any block of $X$ which is not a block of $Y$ and set $T = I - S$, then  $\Delta_{S,T}$ restricts to an injective map on $\nabla_X(\q)$ by part (b) and to the zero map on $\nabla_Y(\q)$ by part (c).
Therefore on $\nabla_X(\q) \cap \nabla_Y(\q)$ the map $\Delta_{S,T}$ is both zero and injective, which occurs only if the intersection is zero.
This proves our claim and we conclude that $\sum_{X \in \Pi[I]} \nabla_X(\q) = \bigoplus_{X \in \Pi[I]} \nabla_X(\q)$.
It is now enough to show that 
\be \label{contain-eq} \p[I] \subset \sum_{X \in \Pi[I]} \nabla_X(\q).\ee
This it true by construction if $I$ is empty, so assume $|I| \geq 1$ and that the desired containment holds whenever $I$ is replaced by a smaller set.
If $I = S\sqcup T$ is a disjoint decomposition
with both $S$ and $T$ nonempty then $\im(\nabla_{S,T})$ is contained in the right side of \eqref{contain-eq}
by induction and part (b).
Since $\r[I]$ is the sum of such images, since $\q[I] = \nabla_{X}(\q) $
for the unique partition $X$ of $I$ with one block,
and since $\p[I] = \q[I] \oplus \r[I]$, the desired containment
\eqref{contain-eq} holds.


Deducing from parts (a), (b), and (c)  that $\Delta$ is cocommutative is straightforward.
For the last assertion in the theorem, 
define for each set partition $X \in \Pi[I]$ a $\kk$-linear map $f_X : \q(X) \to \nabla_X(\q)$ in the following way.
If $I = \varnothing$ so that $\q(X) = \kk$ then we set $f_X = \eta_\varnothing$. 
If $I \neq \varnothing$, then  define $f_X$ as the composition
\[ \q(X) = \bigotimes_{B \in X} \q[B] \xrightarrow{\quad\sim\quad} \q[B_1] \otimes \dots \otimes \q[B_k] \xrightarrow{\quad \nabla_{B_1\dots,B_k}\quad} \nabla_X(\q)\]
where $X = \{ B_1,B_2,\dots,B_k\}$ is an arbitrary ordering of the blocks of $X$.
The commutativity of $\nabla$ ensures that this definition has no dependence on the chosen order for the blocks of $X$, and the
injectivity of $\nabla_{B_1\dots,B_k}$ (which follows by Lemma \ref{injsurj-fact}) implies that $f_X$ is an isomorphism.
%
By part (a), the sums of maps 
\be\label{f_I}
f_I \omdef = \bigoplus_{X \in \Pi[I]} f_X :  \cS(\q)[I] \to \q[I]
\ee
therefore give the components of an isomorphism of vector species $f : \cS(\q) \to \p$.
By  Fact \ref{connected-morph-fact} it follows that $f$ is an isomorphism of connected species, and by 
parts (b) and (c) it follows that $f$ is an isomorphism of connected Hopf monoids.
Since we view $\q$ as a subspecies of $\cS(\q)$ by identifying $\q[I]$ with the subspace $\q(\{I\}) \subset \cS(\q)[I]$, the isomorphism $f$ 
 extends the inclusion $\q \to \p$. It is clear from parts (a) and (b), finally,  that $f$ is the unique such extension.
\end{proof}

\subsection{Strongly self-dual Hopf monoids}\label{mufsd-sect}

A connected monoid $(\p,\nabla)$ is \emph{linearized} 
(in the sense of 
\cite[Section 8.7.2]{species})
if the connected species $\p$  has a basis $\P$ in which the product $\nabla$ is linearized (in the sense of the introduction). 
Recall from the introduction that a connected Hopf monoid 
$\h $ 
is \emph{strongly self-dual} with respect to some  basis  if with respect to that basis, $\h$ is both freely self-dual  and linearized as a monoid. 

%

%

It is convenient to have a word to refer to the system of maps $\P[S]\times \P[T] \to \P[S\sqcup T]$ whose linearizations give the product of a linearized monoid.
In the following definition and throughout this section, $\P$ will denote a set species with linearization $\p =\kk\P$.

\begin{definition}\label{mu-def} A family 
 $\mu = (\mu_{S,T})$
of maps $\mu_{S,T} :   \P[S]\times \P[T] \to \P[S\sqcup T]$ indexed by pairs of disjoint finite sets $S$, $T$ is a \emph{multiplicative system} for $\P$ if
there exists a morphism $\nabla : \p\cdot \p \to \p$ 
whose components $\nabla_{S,T}$ are the linearizations of $\mu_{S,T}$. 
\end{definition}
%

\begin{notation}
If $\mu$ is a multiplicative system then the morphism $\nabla : \p\cdot \p \to \p$ in  
this definition
is
  uniquely determined, and we denote it by $\nabla^\mu$.
  \end{notation}


One can interpret  a multiplicative system as a morphism $\mu : \P\cdot \P \to \P$ where the Cauchy product $\P\cdot \P$ of set species is defined as in \cite[Section 8.7.1]{species}, namely, by replacing the direct sums in Definition \ref{cauchy-def} by disjoint unions and the tensor products by Cartesian products.
We avoid defining a multiplicative system  in this way because the dual notion of a \emph{comultiplicative system} in the next section cannot similarly be interpreted as a morphism $\P \to \P\cdot \P$.

In the examples which follow, given maps $\lambda : S \to S'$ and $\lambda' : T \to T'$ with disjoint domains, we write $\lambda \sqcup \lambda'$ for the unique map $S \sqcup T \to S' \cup T'$ which restricts  to $\lambda$ on $S$ and  to $\lambda'$ on $T$.


\begin{example}\label{mu-ex}
Recall the set species defined at the beginning of Section \ref{prelim-sect}.
For each of those examples, we define a multiplicative system $\mu$ as follows.
\begin{itemize}
\item[(i)] Exponential species: if $\P=\E$ then define $\mu_{S,T}$ as the unique map $\{1_\kk\} \times \{1_\kk\} \to \{1_\kk\}$.

\item[(ii)] Maps: if $\P = \E_C$  then define $\mu_{S,T}(\lambda,\lambda')= \lambda \sqcup \lambda'$ for maps $\lambda : S \to C$ and $\lambda' : T \to C$. 

\item[(iii)] Permutations: if $\P = \fk S$ then define $\mu_{S,T}(\lambda,\lambda') = \lambda \sqcup \lambda'$, so that this is the permutation of $S\sqcup T$ whose cycle representation is given by concatenating the cycles of $\lambda$ and $\lambda'$.

\item[(iv)] Linear orders: if $\P = \L$ then define $\mu_{S,T}(\lambda,\lambda')$
as
 the unique linear order of $S \sqcup T$ restricting to $\lambda$ on $S$ and to $\lambda'$ on $T$ and such that every element of $T$ exceeds every element of $S$.

\item[(v)] Set partitions: if $\P=\Pi$ then define $\mu_{S,T}(\lambda,\lambda') = \lambda \sqcup \lambda'$, so that this is the partition of $S\sqcup T$ whose blocks are the blocks of $\lambda$ and $\lambda'$.


\end{itemize}

\end{example}

For the duration of this section $\mu$ denotes a multiplicative system for  $\P$.
Concerning  such a system, we introduce the following terminology:
\begin{itemize}
\item $\mu$ is \emph{injective} or \emph{surjective} if the maps $\mu_{S,T}$ are always injective or surjective.
\item $\mu$ is \emph{associative} or \emph{commutative} if the morphism $\nabla^\mu$ is associative or commutative. 
\item  $\mu$ is \emph{unital} if $\P[\varnothing] = \{1_\P\}$ is a singleton set and the maps $\mu_{\varnothing,S}$ and $\mu_{S,\varnothing}$ are always the canonical identifications $(1_\P,\lambda) \mapsto \lambda$ and $(\lambda,1_\P)\mapsto \lambda$.


\end{itemize}
All of the multiplicative systems in Example \ref{mu-ex} are injective, associative, and unital. Only systems (i) and (ii) are surjective  while only system (iv) is non-commutative. 
%
We will sometimes 
write $\lambda\cdot\lambda'$ for the image of $(\lambda,\lambda') \in \P[S]\times\P[T]$ under $\mu_{S,T}$ when the system $\mu$ and the sets $S$, $T$ are clear from context. 
Using this notation, we can say that
$\mu$ is associative or commutative if and only if the familiar identities 
$\lambda\cdot(\lambda'\cdot\lambda'') = (\lambda\cdot\lambda')\cdot\lambda'' $ or $\lambda\cdot\lambda' =\lambda'\cdot\lambda$
always hold.

Continuing, we have the following fact. Recall here that the set species $\P$ is \emph{finite} if $\P[S]$ is always a finite set. 

 \begin{fact} \label{mumorph}
 If $\P$ is finite then there 
 is a unique morphism 
 $\Delta^\mu : \p \to \p \cdot \p $ of vector species 
whose components for disjoint finite sets $S$, $T$ have the formula
  \[\Delta^\mu_{S,T}(\gamma) = \sum_{\substack{(\lambda,\lambda') \in \P[S]\times \P[T] \\ \mu_{S,T}(\lambda,\lambda') = \gamma}} \lambda\otimes \lambda'
  \qquad\text{for }
\gamma \in \P[S\sqcup T].
\] 
The morphism $\nabla^\mu$ is associative (respectively, commutative) if and only if $\Delta^\mu$ is coassociative (respectively, cocommutative).
 \end{fact}
 
 \begin{remark}
If $\mu$ is injective then the formula for $\nabla^\mu$ simplifies considerably, since then the sum defining
 $\Delta^\mu_{S,T}(\gamma)$ 
 has at most one summand.
\end{remark}
 
 \begin{proof}
The desired morphism $\Delta^\mu$ is   the  composition 
$\p \xrightarrow{\sim} \p^* \xrightarrow{(\nabla^\mu)^*} (\p\cdot \p)^* \xrightarrow{\sim} \p^* \cdot \p^* \xrightarrow{\sim} \p\cdot \p$
 where the first  map is the  isomorphism $f: \p \to \p^*$ induced by   $\P$ and the last map is $f^{-1}\cdot f^{-1}$. 
 \end{proof}

 The following proposition, whose proof is just a summary of the preceding discussion, describes precisely when a multiplicative system gives rise to a linearized connected monoid.
 
 \begin{proposition}\label{monoid-prop}  
 Let $\mu$ be a multiplicative system for a finite set species $\P$ 
  with linearization $\p=\KK\P$. The following are then equivalent:
 \ben
 \item[(a)] $\mu$ is associative and unital.
 
 \item[(b)] $(\p,\nabla^\mu)$ is a connected monoid.
 
 \item[(c)] $(\p,\Delta^\mu)$ is a connected comonoid.
 \een
 Moreover, if these conditions hold then $(\p,\nabla^\mu)$ and $(\p,\Delta^\mu)$ are isomorphic to each other's duals via the morphism $\p \to \p^*$ induced by $\P$.
\end{proposition}


Suppose $\mu$ and $\mu'$ are multiplicative systems for  set species $\P$ and $\P'$. A \emph{morphism} of multiplicative systems $(\P,\mu) \to (\P',\mu')$ is a morphism of set species $f : \P \to \P'$ 
such that
\[ f_{S\sqcup T}\circ \mu_{S,T} = \mu'_{S,T}\circ (f_S \times f_T)\] for all disjoint  finite sets $S$, $T$.
The proposition implies that $(\P,\mu)\cong (\P',\mu')$ if and only if $(\kk\P,\nabla^\mu) \cong (\kk\P',\nabla^{\mu'})$ as connected monoids 
and also, provided $\P$ and $\P'$ are finite,  $(\kk\P,\Delta^\mu) \cong (\kk\P',\Delta^{\mu'})$  as connected comonoids.


The proposition likewise implies that the triple  $\h= (\kk\P,\nabla^\mu,\Delta^\mu)$ is a  connected Hopf monoid if and only if $\P$ is finite and $\mu$ is associative, unital, and \emph{Hopf self-compatible}$-$where we say that $\mu$ is Hopf self-compatible if  the morphisms $\nabla^\mu$ and $\Delta^\mu$ are Hopf compatible in the sense of Definition \ref{hopf-compat-def}. In this case $\h$ is  strongly self-dual, and every SSD connected Hopf monoid arises as such a triple for some multiplicative system $\mu$.
This observation shows that we may redefine the property of being strongly self-dual by the following more elementary conditions. 

\begin{proposition}\label{redef-prop}
A connected Hopf monoid $(\p,\nabla,\Delta)$ is strongly self-dual in some basis $\P$ for $\p$ if and only if   for all disjoint finite sets $S$, $T$ the following conditions hold:
\ben
\item [(a)]   $\nabla_{S,T}$ restricts to a map $\P[S]\times\P[T] \to \P[S\sqcup T]$.
\item[(b)] $\P[S\sqcup T]$ is a  finite subset of the union of the image of $\nabla_{S,T}$ and the kernel of $\Delta_{S,T}$.

\een
\end{proposition}

\begin{proof}
Let $\h=(\p,\nabla,\Delta)$ be a connected Hopf monoid. 
%
Suppose $\p = \kk\P$ for a  set species $\P$ and (a) and (b) hold,
so that $\P$ is finite and $\nabla = \nabla^\mu$ for some multiplicative system $\mu$ for $\P$. Combining (b)
with Lemma \ref{injsurj-fact}
then shows that $\Delta = \Delta^\mu$.
In detail, let $\gamma \in \P[S\sqcup T]$
and observe that if $\gamma$ does not belong to the image of $\nabla_{S,T}$ then \[\Delta_{S,T}(\gamma) = \Delta^\mu_{S,T}(\gamma) = 0\]
since (b) implies $\gamma \in \ker(\Delta_{S,T})$.
If alternatively $\gamma$ does belong to the image of $\nabla_{S,T}$,
then 
 there is a unique pair $(\lambda,\lambda') \in \P[S]\times \P[T]$ with $\nabla_{S,T}(\lambda\otimes \lambda') = \mu_{S,T}(\lambda,\lambda') = \gamma$
and we have 
\[ \Delta_{S,T}(\gamma) = \Delta^\mu_{S,T}(\gamma) = \lambda\otimes \lambda'\] since $\Delta_{S,T}$ is the left inverse of $\nabla_{S,T}$. 
Hence $\Delta = \Delta^\mu$, which suffices to show that $\h$ is strongly self-dual with respect to $\P$. 
That conversely any SSD connected  Hopf monoid has the given properties follows from the fact that this is true for any connected Hopf monoid of the form $(\kk\P,\nabla^\mu,\Delta^\mu)$.
\end{proof}

It is usually straightforward to detect whether a given multiplicative system is associative or unital, directly from the definitions. It remains to give a simple criterion  to check whether $\mu$ is Hopf self-compatible, and we do this with 
the following theorem. 

\begin{theorem}\label{selfcompat-thm}
Suppose $\P$ is a finite set species and  $\mu$ is a multiplicative system for $\P$ which is associative and unital.
Then $\mu$ is Hopf self-compatible if and only if the following conditions hold:
\begin{itemize}
\item[(a)] $\mu$ is commutative.
\item[(b)] $\mu$ is injective.
\item[(c)] If  $ S\sqcup S' = A \sqcup B$ are two disjoint decompositions of the same finite set 
then 
\[ 
\lambda \in \im(\mu_{A\cap S,B\cap S}) \qquand \lambda' \in \im(\mu_{A\cap S',B\cap S'})
\]
whenever  $(\lambda,\lambda') \in \P[S]\times \P[S']$ such that $\mu_{S,S'}(\lambda,\lambda ') \in \im(\mu_{A,B})$.


%
  \end{itemize}
\end{theorem}

\begin{remark}
The theorem shows that all but system (iv) in Example \ref{mu-ex} are Hopf self-compatible.
\end{remark}

\begin{proof}
Suppose $\mu$ is Hopf self-compatible. Then $(\kk\P,\nabla^\mu,\Delta^\mu)$ is an FSD connected Hopf monoid so $\nabla^\mu$ is commutative by Theorem \ref{coco-cor} and the maps $\nabla^\mu_{S,T}$ are injective by Lemma \ref{injsurj-fact}. We conclude that $\mu$ in commutative and injective.
To show that property (c) holds, suppose $S\sqcup S' = A\sqcup B$ and $(\lambda,\lambda') \in \P[S]\times \P[S']$ such that $\lambda\cdot \lambda ' = \alpha\cdot \beta$ for some $(\alpha,\beta) \in \P[A]\times \P[B]$.
By definition,  
\[(\Delta^\mu_{A,B}\circ \nabla_{S,S'}^\mu)(\lambda\otimes \lambda') = \alpha\otimes \beta \neq 0.\]
On the other hand,  the Hopf compatibility of $\nabla^\mu$ and $\Delta^\mu$ implies $\Delta^\mu_{A,B}\circ \nabla_{S,S'}^\mu = f\circ \tau \circ g$ where  
\[ f = \nabla^\mu_{A\cap S,A\cap S'} \otimes \nabla^\mu_{B\cap S, B\cap S'}\qquand g  =\Delta^\mu_{S\cap A,S\cap B} \otimes \Delta^\mu_{S'\cap A,S'\cap B}\]
and $\tau$ is an appropriate twisting isomorphism. As these are all linear maps, we must have $g(\lambda\otimes \lambda')\neq 0 $, and by the definition of $\Delta^\mu$ this holds precisely when we have both $\lambda \in \im(\mu_{A\cap S,B\cap S}) $ and $ \lambda' \in \im(\mu_{A\cap S',B\cap S'})$.

Conversely suppose that $\mu$ has properties (a), (b), and (c). Let $S\sqcup S' = A\sqcup B$ be two disjoint decompositions of the same finite set and fix $(\lambda,\lambda') \in \P[S]\times \P[S']$. We must show that $\nabla^\mu$ and $\Delta^\mu$ are Hopf compatible, and to do this it suffices to check that $x=y$ where
\[ x\omdef =  (\Delta^\mu_{A,B}\circ \nabla_{S,S'}^\mu)(\lambda\otimes \lambda')
\qquand
y \omdef = (f\circ \tau \circ g)(\lambda\otimes \lambda'),\] with the maps $f$ and $\tau$ and $g$  defined as above. 
There are two cases to consider. First suppose $\lambda\cdot  \lambda' \notin \im(\mu_{A,B})$. Since $\mu$ is associative and commutative, it follows that either
 $\lambda \notin \im(\mu_{A\cap S,B\cap S}) $ or $ \lambda' \notin \im(\mu_{A\cap S',B\cap S'})$, and so we have $x=y=0$
 by the definition of $\Delta^\mu$.
Alternatively, suppose 
 $\lambda\cdot  \lambda'  =  \alpha\cdot \beta$ for some $(\alpha,\beta) \in \P[A]\times \P[B]$. 
Since $\mu$ is injective, our definition of $x$ then reduces to
\[x = \alpha \otimes \beta.\]
Property (c)  ensures that $\lambda =\alpha_S\cdot \beta_S$ and $\lambda' = \alpha_{S'}\cdot \beta_{S'}$
for  some $\alpha_S \in \P[S\cap A]$ and $\beta_S\in \P[S\cap B]$ and $\alpha_{S'}\in \P[S'\cap A] $ and $\beta_{S'} \in \P[S'\cap B]$, and our definition of $y$ likewise reduces to \[y = (\alpha_S \cdot \alpha_{S'}) \otimes (\beta_{S}\cdot \beta_{S'}).\]
Since $\mu$ is associative and commutative, we may compute
\[\ba (\alpha_S \cdot \alpha_{S'}) \cdot (\beta_{S}\cdot \beta_{S'}) &= \alpha_S \cdot( \beta_{S} \cdot \alpha_{S'})\cdot \beta_{S'}
\\&
= \alpha_S \cdot( \alpha_{S'} \cdot \beta_{S})\cdot \beta_{S'} 
\\&
= (\alpha_S \cdot \beta_{S}) \cdot (\alpha_{S'}\cdot \beta_{S'}) = \lambda\cdot \lambda' = \alpha\cdot \beta. 
\ea\]
Since $\mu$ is injective, this identity implies that $\alpha_S\cdot \alpha_{S'} = \alpha$ and $\beta_S \cdot \beta_{S'} = \beta$, and thus  again  $x=y$, which is what we needed to show.
\end{proof}

Recall from the introduction that if $\Q$ is a set species then $\cS(\Q)$ is the species of $\Q$-labeled set partitions.
If $\SetSp$ denotes the category of set species and $\SetSp_+$ is the full subcategory of set species $\Q$ with $\Q[\varnothing]= \varnothing$, then we can view $\cS$ as a faithful functor
\be\label{set-functor-eq} \cS : \SetSp_+ \to \SetSp\ee
by defining $\cS(f)$ for a morphism $f : \Q \to \Q'$ to be the  morphism $\cS(\Q) \to \cS(\Q')$  whose $I$-component acts on labeled partitions by the formula 
$
X \mapsto \{ (B,f_B(\lambda)) : (B,\lambda) \in X\}.
$
It follows  that $\cS(\Q)\cong \cS(\Q')$ whenever $\Q$ and $\Q'$ are isomorphic set species with $\Q [\varnothing] = \Q'[\varnothing] = \varnothing$.

The multiplicative system 
for the species of set partitions $\Pi$
 in Example \ref{mu-ex}(v) admits the following generalization to $\cS(\Q)$. Namely, for any disjoint sets $S$, $T$   denote by 
 $\cup_{S,T}$  the map $   \cS(\Q)[S]\times \cS(\Q)[T]  \to   \cS(\Q)[S\sqcup T] $ given the disjoint union $(X,Y) \mapsto X\sqcup Y$.
The following fact now comes as a corollary to Theorem \ref{selfcompat-thm}.
\begin{corollary}\label{cup-cor}
If  $\Q$ is a finite set species, then  $\cup = (\cup_{S,T})$ forms a multiplicative system for $\cS(\Q)$ which is associative, unital, and Hopf self-compatible,
and the natural isomorphism $\cS(\kk\Q) \xrightarrow{\sim} \kk\cS(\Q)$ defines an isomorphism of connected Hopf monoids $\cS(\kk\Q) \cong (\kk\cS(\Q),\nabla^\cup,\Delta^\cup)$.
\end{corollary}

\begin{proof}
That $\cup$ forms an associative, commutative, and injective multiplicative system follows by inspection. Since $\cS(\Q)[\varnothing]= \{ \varnothing\}$ we likewise find that $\cup$ is unital. To show that $\cup$ is Hopf self-compatible it remains only to check property (c) in Theorem \ref{selfcompat-thm}, and this is straightforward. 
The last part of the corollary is clear from the definition of $\cS(\kk\Q)$.
\end{proof}

Define $\cP(\P,\mu)\subset \P$ as the  subspecies  such that $\cP(\P,\mu)[\varnothing] = \varnothing$ and such that 
if $I$ is nonempty then
 $\cP(\P,\mu)[I]$
 is the set of elements in $\P[I]$ which are not in the image of $\mu_{S,T}$ for any disjoint decomposition $I = S\sqcup T$ with $S$ and $T$  both nonempty. 
  Our notation is justified by the following.
   
\begin{proposition} \label{fsd-prim-prop}
Let  $\mu$ be a multiplicative system for a finite set species $\P$ and set $\Q = \cP(\P,\mu)$.
If $(\kk\P,\nabla^\mu,\Delta^\mu)$ is a connected Hopf monoid, then $\kk\Q$ is its subspecies of primitive elements. 
\end{proposition}

This statement's proof, which we omit, is a simple exercise from the definition of $\Delta^\mu_{S,T}$.

%
%


\begin{example}\label{cP-ex}
Given a set $C$, and let $\X_{C}$ denote the set species with $\X_C[S] = C$ for sets $S$ with one element and $\X_C[S] = \varnothing$ otherwise. Define $\X_C[\sigma]$ for bijections $\sigma$ to always be the identity map. 
We compute  $\cP(\P,\mu)$ for the set species  $\P$ and multiplicative systems $\mu$  in 
Example \ref{mu-ex}:
\ben
\item[(i)]  Exponential species: If $\P=\E$ then $\cP(\P,\mu) = \X$ where $\X\omdef= \X_{\{1\}}$.

\item[(ii)] Maps: if $\P = \E_C$  then $\cP(\P,\mu) \cong \X_{C}$.

\item[(iii)] Permutations: if $\P=\fk S$ then $\cP(\P,\mu)$ is the subspecies of  transitive permutations.

\item[(iv)] Linear orders: if $\P = \L$ then $\cP(\P,\mu) \cong \X$.
\item[(v)] Set partitions: if $\P=\Pi$ then $\cP(\P,\mu) \cong \E$.
\een
\end{example}


We may now prove Theorems \descref{B} and \descref{C} from the introduction.

\begin{theorem}\label{fsd-thm}
Let $\h$ be a connected Hopf monoid which is commutative and cocommutative. 
Then $\h$ is  
strongly self-dual
 if and only if its subspecies of primitive elements has a finite basis.
 \end{theorem}
 
 \begin{proof}
 Let $\q=\cP(\h)$.
 If $\h$ is strongly self-dual then $\q$ has a finite basis by Proposition \ref{fsd-prim-prop}. 
Alternatively, if $\q$ has a finite basis $\Q$, then  $\h \cong \cS(\q) \cong (\kk\cS(\Q),\nabla^\cup,\Delta^\cup)$ by Theorem \ref{coco-thm} and Corollary \ref{cup-cor}, so $\h$ is strongly self-dual.
 \end{proof}

Theorem \ref{selfcompat-thm} gave a set of ``local'' properties characterizing a multiplicative system which is associative, unital, and Hopf self-compatible. The following  result, alternatively, gives  a ``global'' characterization of such systems. 
Here, observe that there is a natural inclusion $\Q \subset \cS(\Q)$ given by identifying $\Q$ with the subspecies of $\Q$-labeled partitions with exactly one block.

\begin{theorem} \label{muFSD-mainthm}
Suppose  $\mu$ is a multiplicative system for a finite set species $\P$ and $\Q = \cP(\P,\mu)$.
If $\mu$ is associative, unital, and Hopf self-compatible, then there is a unique 
isomorphism of multiplicative systems
\[f^\mu : (\cS(\Q),\cup)\xrightarrow{\sim} (\P,\mu)\]
making the diagram 
\[
\begin{diagram}
\cS(\Q) && \rTo^{f^\mu} && \P \\
& \luTo &&\ruTo\\
&& \Q
 \end{diagram}
 \]
commute, where the diagonal arrows are the natural inclusions of set species.
%
%
%
%
\end{theorem}

We deduce this result as a corollary to Theorems \ref{coco-thm} and \ref{coco-cor}, but one could also give a  proof using only Theorem \ref{selfcompat-thm}.

\begin{proof}
To see the uniqueness of $f^\mu$, suppose
$g^\mu $ is another isomorphism   $(\cS(\Q),\cup) \xrightarrow{\sim} (\P,\mu)$ making the diagram in the theorem commute. Then $(f^\mu)^{-1} \circ g^\mu$ is an automorphism of the multiplicative system $(\cS(\Q),\cup)$ which restricts to the identity on the subspecies $\Q$. Since every $\Q$-labeled set partition is the union of labeled partitions with only one block (which we identity with $\Q$), the composition  $(f^\mu)^{-1} \circ g^\mu$ must be the identity map
so $f^\mu =g^\mu$.

To construct $f^\mu$, 
assume $\mu$ is associative, unital and Hopf self-compatible, and for each 
 finite set $I$    define a map $f^\mu_I : \cS(\Q)[I] \to \P[I]$ as follows. When $I$ is empty both $\cS(\Q)[\varnothing]$ and $\P[\varnothing]$ are  singleton sets so we define $f^\mu_\varnothing$ as the unique map between them. When $I \neq \varnothing$
and  $X \in \cS(\Q)[I]$ is a labeled set partition with $k$ blocks, we 
set
\[f^\mu_I(X) = \lambda_1\cdot \lambda_2\cdots \lambda_k \]
where  $(B_1,\lambda_1)$, $(B_2,\lambda_2)$, \dots, $(B_k,\lambda_k)$ is an arbitrary ordering of the blocks of $X$. 
Here we use the notation $\lambda \cdot \lambda'$ to denote the image of $(\lambda,\lambda') \in \P[S]\times \P[T]$ under $\mu_{S,T}$. It makes no difference how we parenthesize this product (and thus we have omitted all  parentheses) since $\mu$ is associative. Likewise, the product has no dependence on how we order the blocks of $X$ since $\mu$ is commutative by Theorem \ref{selfcompat-thm}.
 
The linearization of $f^\mu_I$ is precisely the map \eqref{f_I} defined  in the proof of Theorem \ref{contained-thm} with $\h = (\kk\P, \nabla^\mu,\Delta^\mu)$. Hence 
 the maps $f^\mu_I$ are the components of a morphism $f^\mu : \cS(\Q) \to \P$
 whose linearization is an isomorphism of connected Hopf monoids $(\kk\cS(\Q),\nabla^\cup,\Delta^\cup) \xrightarrow{\sim} (\kk\P,\nabla^\mu,\Delta^\mu)$. It follows that $f^\mu$ must be itself an isomorphism of multiplicative systems.
\end{proof}

 \subsection{Linearly self-dual Hopf monoids}
 \label{pifsd-sect}
 
Again let $\P$ denote a fixed set species and with linearization $\p = \kk\P$. Following \cite[Section 8.7.2]{species}, we say that 
a connected comonoid $(\p,\Delta)$ is \emph{linearized} with respect to $\P$ if the coproduct $\Delta$ is linearized in this basis (in the sense of the introduction). 
Recall that a connected Hopf monoid $\h$ is \emph{linearly self-dual} with respect to some basis if in that basis, $\h$ is both freely self-dual and linearized as a comonoid.
We classify such Hopf monoids in this section.

Towards this end, we have several facts and definitions dual to the ones in the preceding section. 
Complementing Definition \ref{mu-def}, we have the following problem:

\begin{definition}\label{pi-def} A family 
 $\pi = (\pi_{S,T})$
of maps $\pi_{S,T} :   \P[S\sqcup T] \to \P[S]\times \P[T]$ indexed by pairs of disjoint finite sets $S$, $T$ is a \emph{comultiplicative system} for $\P$ if
there exists a morphism $\Delta : \p \to  \p\cdot \p $ 
whose components $\Delta_{S,T}$ are the linearizations of $\pi_{S,T}$.
\end{definition}

\begin{notation}
If $\pi$ is a comultiplicative system then the morphism $\Delta : \p \to \p\cdot \p $ in 
Definition \ref{pi-def}
 is uniquely determined, and we denote it by $\Delta^\pi$.
\end{notation}

\begin{example}\label{pi-ex}
For each  species in Example \ref{mu-ex} we define a comultiplicative system $\pi$ as follows:
\begin{itemize}
\item[(i)] Exponential species: if $\P = \E$ then
define $\pi_{S,T}$ as the unique map $\{1_\kk\} \to \{1_\kk\}\times \{1_\kk\}$.

\item[(ii)] Maps: if $\P = \E_C$ then define $\pi_{S,T}(\lambda)= (\lambda|_S, \lambda|_T)$ for any map $\lambda : S\sqcup T \to C$.

\item[(iii)] Permutations: if $\P = \fk S$ then define $\pi_{S,T}(\lambda) = (\lambda',\lambda'')$ for permutations $\lambda \in \fk S[S\sqcup T]$, where $\lambda' $ is the permutation of $S$ such that if $i \in S$ then $\lambda'(i)$ is the first element of the sequence $\lambda(i),\ \lambda^2(i),\ \lambda^3(i),\ \dots$ belonging to $S$, and $\lambda'' \in \fk S[T]$ is defined likewise.

\item[(iv)] Linear orders: if $\P=\L$ then define $\pi_{S,T}(\lambda) = (\lambda|_S,\lambda|_T)$ where $\lambda|_S$ and $\lambda|_T$ are the orders of $S$ and $T$ induced by the linear order $\lambda \in \L[S\sqcup T]$.

\item[(v)] Set partitions: if $\P = \Pi$ then define $\pi_{S,T}(\lambda) = (\lambda|_S,\lambda|_T)$ where $\lambda|_S$ denotes the partition of $S$ whose blocks are the nonempty intersections of the form $B \cap S$ with $B $ ranging over all blocks of  $ \lambda \in \Pi[S]$, and  $\lambda|_T$ is defined likewise.

\end{itemize}
\end{example}

For the rest of this section $\pi$ will denote a fixed comultiplicative system for  $\P$.
Concerning  such a system, we adapt the following terminology  from  the last section:
\begin{itemize}
\item $\pi$ is \emph{injective} or \emph{surjective} if the maps $\pi_{S,T}$ are always injective or surjective.

\item $\pi$ is \emph{coassociative} or \emph{cocommutative} if the morphism $\Delta^\mu$ is coassociative or cocommutative.

\item  $\pi$ is \emph{counital} if $\P[\varnothing] = \{1_\P\}$ is a singleton set and the maps $\pi_{\varnothing,S}$ and $\pi_{S,\varnothing}$ are always the canonical identifications $\lambda\mapsto (1_\P,\lambda)$ and $\lambda\mapsto (\lambda,1_\P)$.


\end{itemize}
All of the comultiplicative systems in Example \ref{pi-ex} are surjective, coassociative, cocommutative, and counital. Only systems (i) and (ii) are injective. 
In turn, we have the following comultiplicative version of Fact \ref{mumorph}.

 \begin{fact} \label{pimorph}
 If $\P$ is finite then there 
 is a unique morphism 
 $\nabla^\pi : \p\cdot \p \to \p  $ of vector species 
whose components for disjoint finite sets $S$, $T$ have the formula
  \[\nabla^\pi_{S,T}(\alpha \otimes \beta) =\sum_{\substack{\lambda \in \P[S\sqcup T] \\ \pi_{S,T}(\lambda) = (\alpha,\beta)}} \lambda
  \qquad\text{for }
(\alpha,\beta) \in \P[S]\times\P[T].
\] 
The morphism $\nabla^\pi$ is associative (respectively, commutative) if and only if $\Delta^\pi$ is coassociative (respectively, cocommutative).
\end{fact}

 \begin{proof} 
 $\nabla^\pi$ is  the composition
$\p\cdot \p \xrightarrow{\sim} \p^* \cdot \p^* \xrightarrow{\sim} (\p\cdot \p)^* \xrightarrow{(\Delta^\pi)^*} \p^* \xrightarrow{\sim} \p$
 where again the first and last maps derive from the isomorphism $\p \cong \p^*$ induced by the basis $\P$.
 \end{proof}
 


%
%
%


We next have this analogue of Proposition \ref{monoid-prop}, whose proof we likewise omit.

 \begin{proposition}\label{comonoid-prop} Let $\pi$ be a comultiplicative system for a finite set species $\P$ with linearization $\p=\KK\P$. The following are then equivalent: 
 \begin{itemize}
 \item[(a)] $\pi$ is coassociative and counital. 
 \item[(b)]  $(\p,\nabla^\pi)$ is a connected monoid.
 \item[(c)] $(\p,\Delta^\pi)$ is a connected comonoid.
 \end{itemize}
Moreover, if these conditions hold then $(\p,\nabla^\pi)$ and $(\p,\Delta^\pi)$ are isomorphic to each other's duals via the morphism $\p\to \p^*$ induced by $\P$.
\end{proposition}

%

Our remarks concerning this proposition are similar to those following Proposition \ref{monoid-prop}.
Suppose $\pi$ and $\pi'$ are comultiplicative systems for  set species $\P$ and $\P'$. 
A \emph{morphism} of comultiplicative systems $(\P,\pi) \to (\P',\pi')$ is a 
morphism of set species $f : \P \to \P'$ such that 
\[
(f_{S}\times f_{T})\circ \pi_{S,T} = \pi'_{S,T}\circ f_{S\sqcup T}
\] for all disjoint  finite sets $S$, $T$.
The proposition implies that $(\P,\pi) \cong (\P',\pi')$   if and only if $(\kk\P,\Delta^\pi) \cong (\kk\P',\Delta^{\pi'})$ as connected comonoids and also, provided $\P$ and $\P'$ are finite, $(\kk\P,\nabla^\pi) \cong (\kk\P',\nabla^{\pi'})$ as connected monoids.

We conclude from the proposition that   $(\kk\P,\nabla^\pi,\Delta^\pi)$ is a  connected Hopf monoid if and only if $\P$ is finite and $\pi$ is coassociative, counital, and \emph{Hopf self-compatible}$-$where again we say that $\pi$ is Hopf self-compatible if the morphisms $\nabla^\pi$ and $\Delta^\pi$ are Hopf compatible. A connected Hopf monoid is linearly self-dual precisely when it is equal to such a triple for some choice of $\P$ and $\pi$.
The following result shows that Hopf self-compatibility for comultiplicative systems is equivalent to a simpler condition than in the multiplicative case.

\begin{theorem}\label{pi-selfcompat-lem}
Suppose $\P$ is a finite set species and $\pi$ is a comultiplicative system for $\P$ which is 
 coassociative and counital. Then 
 $\pi$ is  Hopf self-compatible if and only if $\pi$ is cocommutative and bijective.
\end{theorem}

\begin{proof}
Write $\p =\kk\P$ and suppose $\pi$ is Hopf self-compatible.
Then, by Lemma \ref{injsurj-fact}, the composition $\Delta^\pi_{S,T} \circ \nabla^\pi_{S,T}$ is the identity map on $\p[S]\otimes \p[T]$ for all disjoint finite sets $S$, $T$. By the definition of $\Delta^\pi$ and $\nabla^\pi$, however, we have 
$(\Delta^\pi_{S,T} \circ \nabla^\pi_{S,T})(\alpha\otimes \beta) = k \cdot \alpha\otimes\beta$ 
for $(\alpha,\beta) \in \P[S]\times \P[T]$, where $k$ is the cardinality of $\{ \lambda \in \P[I] : \pi_{S,T}(\lambda) = (\alpha,\beta)\}$.
Since $\kk$ has characteristic zero, it follows that always $k=1$ so $\pi$ is bijective.

Next, let $S$ and $S'$ be disjoint finite sets and choose $\lambda \in \P[S\sqcup S']$. Suppose $\pi_{S,S'}(\lambda) = (\alpha,\beta)$ and $\pi_{S',S}(\lambda) = (\beta',\alpha')$. To show that $\pi$ is cocommutative it suffices to show that $\alpha = \alpha'$ and $\beta=\beta'$. Towards this end,   note that since $\pi$ is bijective, we have $\alpha\otimes\beta = (\Delta^\pi_{S,S'} \circ \nabla^\pi_{S',S})(\beta'\otimes \alpha')$. Observe that taking $R=S'$ and $R'=S$ in Definition \ref{hopf-compat-def}
gives
\[ \Delta^\pi_{S,S'} \circ \nabla^\pi_{S',S} =  (\nabla^\pi_{\varnothing,S}\otimes \nabla^\pi_{S',\varnothing}) \circ \tau\circ  (\Delta^\pi_{\varnothing,S'}\otimes \Delta^\pi_{S,\varnothing})\]
where $\tau$ is the natural twisting isomorphism $\p[\varnothing] \otimes \p[S']\otimes \p[S]\otimes \p[\varnothing] \xrightarrow{\sim} \p[\varnothing] \otimes \p[S]\otimes \p[S']\otimes \p[\varnothing]$.
Since $\pi$ is counital, the set $\P[\varnothing]$ consists of  a single element $1_\P$ and we have
\[  (\Delta^\pi_{\varnothing,S'}\otimes \Delta^\pi_{S,\varnothing})(\beta'\otimes \alpha') = 1_\P\otimes \beta'\otimes \alpha'\otimes 1_\P\]
 and likewise
\[  (\nabla^\pi_{\varnothing,S}\otimes \nabla^\pi_{S',\varnothing})(1_\P\otimes \alpha'\otimes \beta'\otimes 1_\P) = \alpha'\otimes \beta'.\]
Combining these identities, we get $\alpha\otimes \beta = \alpha'\otimes \beta'$ in $\p[S]\otimes\p[S']$, which allows us to conclude that $\alpha=\alpha'$ and $\beta=\beta'$. Thus $\pi$ is also cocommutative.

To prove the converse,
now suppose that $\pi$ is  cocommutative and bijective (in addition to being coassociative and counital). Let $\mu =  (\mu_{S,T})$ be the family of inverse maps   
$\mu_{S,T} = \pi_{S,T}^{-1} : \P[S]\times \P[T] \to \P[S\sqcup T].$
Then  $\mu$ is  a multiplicative system for $\P$, which is associative, commutative, unital, and bijective. 
Moreover, one checks from the definitions that $\nabla^\mu = \nabla^\pi$ and $\Delta^\mu  = \Delta^\pi$. 
These morphisms are Hopf compatible since $\mu$ is Hopf self-compatible by Theorem \ref{selfcompat-thm}; in particular, property (c) holds automatically when $\mu$ is surjective. Thus $\pi$ is Hopf self-compatible.
%
\end{proof}


 In the  preceding proof we observed that if $\h=(\kk\P,\nabla^\pi,\Delta^\pi)$ is a connected Hopf monoid then $\h = (\kk\P,\nabla^\mu,\Delta^\mu)$ 
 for the bijective multiplicative system $\mu$ given by taking the inverses of the components of $\pi$.
  Thus, we have the following corollary, which was mentioned in Theorem \descref{D}.
  
%
%

\begin{corollary}\label{deltanabla-cor}

If a connected Hopf monoid is linearly self-dual with respect to some basis then it is also strongly self-dual (and hence linearized) with respect to that basis.
\end{corollary}

Recall that if $C$ is a set then $\E_{C}$ is the set species such that $\E_{C}[S]$ is the set of maps $\lambda: S \to C$. Let $\rho= (\rho_{S,T})$ denote the comultiplicative system for 
this set species given by restriction of maps as in Example \ref{pi-ex}(ii),
so that  $\rho_{S,T}(\lambda) = (\lambda|_S,\lambda|_T)$ for any map $\lambda \in \E_C[ S\sqcup T]$ and disjoint finite sets $S$, $T$. 
The following corollary follows from Theorem \ref{pi-selfcompat-lem} by inspection.

\begin{corollary} \label{rho-cor} Provided that $C$ is a finite set, the system of maps $\rho = (\rho_{S,T})$ forms a comultiplicative system for $\E_C$ which is coassociative, counital, and Hopf self-compatible.
\end{corollary} 


The next theorem, which was promised in the introduction, shows that $\rho$ is up to isomorphism the only comultiplicative system for which the preceding corollary holds.

\begin{theorem}\label{pi-mainthm}
Suppose $\pi$ is a comultiplicative system for a  finite set species $\P$ and $C = \P[1]$.
If $\pi$ is coassociative, counital, and Hopf self-compatible, 
then there exists a unique isomorphism of comultiplicative systems
$f^\pi : (\E_{C}, \rho)\xrightarrow{\sim} (\P,\pi)$ such that $f^\pi_{\{1\}}(\lambda) = \lambda(1)$ for all $\lambda \in \E_C[1]$.
\end{theorem}

\begin{proof}
Assume $\pi$ is coassociative, counital, and Hopf self-compatible, so that $\pi$ is cocommutative and bijective by Theorem \ref{pi-selfcompat-lem}.
Let $\mu$ be the multiplicative system for $\P$ with $\mu_{S,T} = \pi_{S,T}^{-1}$
and
let $\nu$ be the multiplicative system for $\E_C$ given in Example \ref{mu-ex}(ii).
Then both $\mu$ and $\nu$ are associative, unital, and Hopf self-compatible, 
so 
if 
$\Q = \cP(\P,\mu)$ and $\R = \cP(\E_C,\nu)$
then
by Theorem \ref{muFSD-mainthm}
  there are isomorphisms of multiplicative systems
\[
  f^\mu : (\cS(\Q),\cup) \xrightarrow{\sim} (\P,\mu)
\qquand
  f^{\nu} : (\cS(\R),\cup) \xrightarrow{\sim} (\E_C,\nu)
 \]
 commuting with the natural inclusions of $\Q$ and $\R$ respectively.
%
%
%
%
%
%

One computes that $\Q \cong \R \cong \X_C$ where $\X_C$ is the set species defined in Example \ref{cP-ex}. In particular, one checks that there is a unique isomorphism $g :\R \xrightarrow{\sim} \Q$ such that $g_{\{1\}}(\lambda) =\lambda(1)$ for all $ \lambda \in \E_C[1] = \R[1]$.
%
Viewing $\cS$ as the functor \eqref{set-functor-eq}, we obtain an isomorphism of multiplicative systems 
$\cS(g) : (\cS(\R),\cup) \xrightarrow{\sim} (\cS(\Q),\cup)$
commuting with $g: \R \to \Q$ and the natural inclusions $\R\to \cS(\R)$ and $\Q \to \cS(\Q)$.
It follows that the composition
 \[f^\pi \omdef = f^\mu\circ \cS(g) \circ (f^\nu)^{-1}\] is an isomorphism of multiplicative systems $(\E_C,\nu) \xrightarrow{\sim} (\P,\mu)$ such that $f^\pi_{\{1\}}(\lambda) = \lambda(1)$ for all $\lambda \in \E_C[1]$. 
Since $\nu$ and $\mu$ are obtained from $\rho$ and $\pi$ by taking inverses, $f^\pi$ is also an isomorphism of comultiplicative systems $(\E_C,\rho) \xrightarrow{\sim} (\P,\pi)$.

To show that $f^\pi$ is the unique such isomorphism, 
%
suppose $f $ is another isomorphism  with the same properties. Then $f^{-1} \circ f^\pi$ is an automorphism of the multiplicative system $(\E_C,\nu)$ which is the identity on $\E_C[1]$ and hence on $\E_C[S]$ for all singleton sets $S$. Since every $\lambda \in \E_C[S]$ is the product under $\nu$ of the maps $\lambda|_{\{i\}} \in \E_C[\{i\}]$ for $i \in S$, it follows that $f^{-1} \circ f^\pi = \id$.
\end{proof}


%


To finish this section, we prove Theorem \descref{E} the introduction. Recall that $\cP(\h)$ denotes the subspecies of primitive elements in a Hopf monoid.

\begin{theorem}\label{pi-lastthm}
Let $\h$ be a  connected Hopf monoid which is commutative and cocommutative. Then $\h$ is linearly self-dual if and only if 
$\cP(\h)[1]$ is finite-dimensional and $\cP(\h)[n] = 0$ for all  $n\neq 1$.
\end{theorem}

\begin{proof}
Let $\h \in \cocoHopf(\Sp^\circ)$ and set $\q = \cP(\h)$. First suppose $\h$ is linearly self-dual.
Then, from the proof of Theorem \ref{pi-mainthm}, $\h $ is  isomorphic  to the strongly self-dual connected Hopf monoid $(\kk\cS(\X_C), \nabla^\cup,\Delta^\cup)$ for a finite set $C$. By Proposition \ref{fsd-prim-prop}, therefore, $\q \cong \kk \X_C$ and so $\q$ has the desired properties.

Conversely suppose $\q[1]$ is finite-dimensional and $\q[n] =0$ for all $n>1$. Then $\q[S] = 0$ for all finite sets $S$ with $|S| \neq 1$, which means that if  $C$ is any basis for the vector space $\q[1]$, then $\q \cong \kk\X_C$ and, further,  every basis for $\q$ is a set species isomorphic to $\X_C$.   By Theorems \ref{muFSD-mainthm} and \ref{fsd-thm}, we therefore
have
\[
\h \cong 
 (\kk\cS(\X_C), \nabla^\cup,\Delta^\cup)\cong (\kk\E_C, \nabla^\rho,\Delta^\rho),
\]
 the second isomorphism following from the observation 
that 
if $\mu$ is the multiplicative system for $\E_C$ with
$\mu_{S,T} = \rho_{S,T}^{-1}$ then $(\cS(\X_C),\cup) \cong (\E_C,\mu)$.
Since $C$ is finite, the last triple is linearly self-dual, and hence so is $\h$.
\end{proof}

%


\section{Structure of linearized Hopf monoids}
\label{linear-sect}

A connected Hopf monoid is \emph{linearized} if with respect to some basis both its product and coproduct are linearized. 
The main results of this section show that  connected Hopf monoids which are commutative, cocommutative, and linearized are equivalent  to connected Hopf monoids which are strongly self-dual with respect to a basis equipped with a certain partial order. 

\subsection{Partial orders on species}
\label{order-sect}

%

Throughout this section $\P$ denotes a finite set species with a multiplicative system $\mu$ and a comultiplicative system $\pi$. 
Above, we studied Hopf monoids of the form $(\kk\P,\nabla^\mu,\Delta^\mu)$ and $(\kk\P,\nabla^\pi,\Delta^\pi)$, and 
we turn our attention here to the remaining variations on this theme, namely, 
Hopf monoids of the form $(\kk\P,\nabla^\mu,\Delta^\pi)$ and $(\kk\P,\nabla^\pi,\Delta^\mu)$.

Let us recall  necessary and sufficient  conditions for these triples to be connected Hopf monoids.
Say that $\mu$ and $\pi$ are \emph{Hopf compatible} if the morphisms $\nabla^\mu$ and $\Delta^\pi$ are Hopf compatible in the sense of Definition \ref{hopf-compat-def}. The triple $(\kk\P,\nabla^\mu,\Delta^\pi)$ is then a connected Hopf monoid if and only if 
\begin{itemize}
\item$(\kk\P,\nabla^\mu)$ is a connected monoid (i.e., $\mu$ is associative and unital).
\item $(\kk\P,\Delta^\pi)$ is a connected comonoid (i.e., $\pi$ is coassociative and counital).
\item $\mu$ and $\pi$ are Hopf compatible.
\end{itemize}
When these conditions hold, the connected Hopf monoid $(\kk\P,\nabla^\mu,\Delta^\pi)$ is evidently linearized, 
and any connected Hopf monoid which is linearized has this form for some choice of $\P$, $\mu$, and $\pi$.

\begin{example} The triple
$(\kk\P,\nabla^\mu,\Delta^\pi)$ is a  connected Hopf monoid whenever    $\P$ is one of the set species $\E$, $\E_C$, $\fk S$, $\L$, or $\Pi$, and  $\mu$ and $\pi$ are given as in
Examples \ref{mu-ex} and \ref{pi-ex}.
\end{example}

We may characterize when the systems $\mu$ and $\pi$ are Hopf compatible more directly as follows. Take the diagram in Definition \ref{hopf-compat-def}
and replace the symbols $\p$, $\otimes$, $\nabla$, $\Delta$ with $\P$, $\times$, $\mu$, $\pi$ respectively. Then $\mu$ and $\pi$ are Hopf compatible if and only if this modified diagram always commutes.
From this, or using a duality argument, it is straightforward to check that if
$\mu$ and $\pi$ are Hopf compatible then the morphisms $\nabla^\pi$ and $\Delta^\mu$ are also Hopf compatible.
%
%
%
%
From Propositions \ref{monoid-prop} and \ref{comonoid-prop} we then obtain the following fact:
\begin{fact}\label{hopfdual-fact}
 If either triple $(\kk\P,\nabla^\mu,\Delta^\pi)$ or 
$(\kk\P,\nabla^\pi,\Delta^\mu)$ is a connected Hopf monoid then both are,
and in this case each is isomorphic to the other's dual via the induced morphism $\kk\P \to (\kk\P)^*$.
\end{fact}

For connected Hopf monoids which are commutative as well as linearized, much of the structure theory in Section \ref{mufsd-sect} carries over as a result of the following lemma.
Here we write $\Hopf(\Sp^\circ)$ for the category of connected Hopf monoids.

\begin{lemma}\label{comcomp-lem} 
Assume $(\kk\P,\nabla^\mu,\Delta^\pi) \in \Hopf(\Sp^\circ)$. Then
 $\mu$ is commutative if and only if $\mu$ is Hopf self-compatible.
\end{lemma}

\begin{proof}
One direction of this result follows from Theorem \ref{selfcompat-thm}. For the other direction, suppose that $\mu$ is commutative.
The morphism  $\nabla^\mu$ is injective by Lemma \ref{injsurj-fact} and this implies that $\mu$ is injective.
To show that $\mu$ is Hopf self-compatible, therefore, 
all we need to do is check that $\mu$ has property (c) in Theorem \ref{selfcompat-thm}. 
To this end, suppose $S\sqcup S' = A\sqcup B$ are two disjoint decompositions of the same finite set and $(\lambda,\lambda') \in \P[S]\times \P[S']$ such that $\mu_{S,S'}(\lambda,\lambda') = \mu_{A,B}(\alpha,\beta)$ for some $(\alpha,\beta) \in \P[A]\times\P[B]$.
Then 
\[ \lambda\otimes \lambda' = (\Delta^\pi_{S,S'}\circ \nabla^\mu_{A,B})(\alpha\otimes\beta)
\]
since $\Delta^\pi_{S,S'}\circ \nabla^\pi_{S,S'}$ is the identity map by Lemma \ref{injsurj-fact}.
Applying the definition of  Hopf compatibility (and also of the morphisms $\nabla^\mu$ and $\Delta^\pi$)
to the composition 
$\Delta^\pi_{S,S'}\circ \nabla^\mu_{A,B}$
gives alternatively
\[  (\Delta^\pi_{S,S'}\circ \nabla^\mu_{A,B})(\alpha\otimes\beta)
 = \mu_{A\cap S,B\cap S}(\lambda_1,\lambda_2) \otimes \mu_{A\cap S',B\cap S'}(\lambda'_1,\lambda'_2)\]
 where $\lambda_i$ and $\lambda'_i$ are the elements such that $\pi_{A\cap S, A\cap S'}(\alpha) = (\lambda_1,\lambda_1')$ and $\pi_{B\cap S, B\cap S'}(\beta) = (\lambda_2,\lambda'_2)$.
 Combining the preceding equations shows that $\lambda \in \im(\mu_{A\cap S,B\cap S})$ and $\lambda' \in \im(\mu_{A\cap S',B\cap S'})$ as required.
 \end{proof}

Thus, when the linearized Hopf monoid $\h=(\kk\P,\nabla^\mu,\Delta^\pi)$ is  commutative 
we may identify $\P$ with some species of labeled set partitions by Theorem \ref{muFSD-mainthm}.
When $\h$ is also cocommutative, there 
is a natural partial order  on $\P$ which generalizes the refinement partial order on  set partitions. (Recall a set partition $X$ \emph{refines} another set partition $Y$  if each block of $Y$ is a union of blocks of $X$. If we write $X \leq Y$ in this situation, then $\leq$ is a partial order on the collection of partitions of a fixed set.)
We introduce this order in the next theorem, 
but 
%
 before considering this, we require some notation and a technical lemma.
 
 \begin{notation}
 Assume $\mu$ is associative and $\pi$ is coassociative. Then, for any pairwise disjoint decomposition $I = S_1 \sqcup \dots \sqcup S_k$ of a finite set, iterating the components of $\mu$ and $\pi$ give rise to   maps which we denote  \[ 
 \mu_{S_1,\dots,S_k} : \P[S_1]\times \dots \times \P[S_k] \to \P[I] 
 \qquad
  \pi_{S_1,\dots,S_k} :\P[I] \to \P[S_1]\times \dots \times \P[S_k].
 \]
Formally, these are the unique maps  whose linearizations are the linear maps $\nabla^\mu_{S_1,\dots,S_k}$ and $\Delta^\pi_{S_1,\dots,S_k}$ given by \eqref{nablam} and \eqref{deltam}.
When $k=1$  both  maps are the identity. 
\end{notation}

The maps just defined are the ones indicated without subscripts in the  diagram in the following lemma, but our primary use of this notation will begin with the statement of Theorem \ref{order-thm}.

\begin{lemma}\label{preorder-lem}
Suppose $\mu$ is associative and commutative and $\pi$ is coassociative and cocommutative. 
If $\mu$ and $\pi$ are Hopf compatible, then the diagram 
\[
\begin{diagram}
\prod_{i=1}^k \P[R_i] &\rTo^{\mu}& & \P[I] &&\rTo^{\pi} & \prod_{j=1}^\ell \P[S_j] \\ 
\\ 
\dTo^{\prod_{i=1}^k\pi}&&  &&  && \uTo_{\prod_{i=1}^\ell \mu} \\ 
\\
\prod_{i=1}^k \(\prod_{j=1}^\ell \P[R_i\cap S_j] \)&&& \rTo^{\sim} &&& \prod_{j=1}^\ell \(\prod_{i=1}^k\P[S_j\cap R_i]\)
 \end{diagram}
 \]
commutes for any pairwise disjoint decompositions $I =R_1\sqcup \dots \sqcup R_k = S_1\sqcup \dots \sqcup S_\ell$. 
 \end{lemma}

\begin{proof}
The proof of this lemma, by induction on $k+\ell$ using Definition \ref{hopf-compat-def}, is  straightforward  but cumbersome to write down in detail, and we leave it as an exercise for the reader.
\end{proof}

Recall from the introduction that a \emph{partial order} on a set species $\P$ is transitive subspecies $\{<\}\subset \P\times \P$ such that if $(a,b) \in \{<\}[I]$ then $a\neq b$.
When $\{ <\}$ is a partial order on $\P$ and $\lambda,\lambda' \in \P[I]$ we write 
\[ \lambda < \lambda' \qquad\text{to indicate}\qquad (\lambda,\lambda') \in \{<\}[I].\]
In turn, we write $\lambda \leq \lambda'$ to mean $\lambda = \lambda'$ or $\lambda < \lambda'$.
With respect to this notation, $<$ then corresponds to the usual notion of a partial order on each set $\P[I]$.
The condition that $\{<\}$ forms a subspecies of $\P\times \P$ is equivalent to requiring that $\P[\sigma](\lambda) < \P[\sigma](\lambda')$ whenever $\lambda < \lambda'$, for all bijections $\sigma : I \to I'$.

These definitions are all at the service of the following result, given as Theorem \descref{G} in the introduction.
Here we write $\cocoHopf(\Sp^\circ)$ for the full subcategory of connected Hopf monoids which are commutative and cocommutative.

\begin{theorem}\label{order-thm}
Assume 
$(\kk\P,\nabla^\mu,\Delta^\pi) \in \cocoHopf(\Sp^\circ)$.
Define $\{ \prec\}[I]$ for each finite set $I$ as the set of pairs $(\lambda,\lambda') \in \P[I]\times \P[I]$ with $\lambda \neq \lambda'$ such that
\[ \lambda = \mu_{S_1,\dots,S_k}\circ \pi _{S_1,\dots,S_k}(\lambda')\]
for some pairwise disjoint decomposition $I = S_1\sqcup \dots \sqcup S_k$. Then $\{\prec\}$ is a partial order on $\P$. Moreover, $\{\prec\}$
 is the minimal transitive subspecies of $\P\times \P$ such that $(\lambda,\lambda') \in \{\prec\}[I]$   whenever $\lambda,\lambda' \in \P[I]$ are distinct and $\lambda = \mu_{S,T}\circ \pi_{S,T} (\lambda')$ for some disjoint decomposition $I = S\sqcup T$.
\end{theorem}

\begin{remark} 
The word \emph{minimal} in this statement refers to minimality in the ordering of the subspecies of $\P\times \P$ by containment. The family of subspecies of $\P\times \P$ with the properties described in the last part of the theorem is closed under intersections, and so it makes sense to speak of the (unique) minimal subspecies in this family.
\end{remark}

\begin{proof}
That $\{\prec\}$ is at least a subspecies of $\P\times \P$ follows from the definition of a (co)multiplicative system. To show transitivity, suppose $(\lambda,\lambda') \in \{\prec\}[I]$ and also $(\lambda',\lambda'') \in \{\prec\}[I]$. By definition, there are then two pairwise disjoint decompositions $I = R_1 \sqcup \dots \sqcup R_k = S_1 \sqcup \dots \sqcup S_\ell$ such that 
\[  \lambda = \mu_{R_1,\dots,R_k}\circ \pi _{R_1,\dots,R_k}(\lambda')
 \qquand
 \lambda'=   \mu_{S_1,\dots,S_\ell}\circ \pi _{S_1,\dots,S_\ell}(\lambda'').
\]
On applying Lemma \ref{preorder-lem} to the composition $\pi _{R_1,\dots,R_k}\circ \mu_{S_1,\dots,S_\ell}$,
we
obtain
\[ \lambda =
 \mu_{R_1,\dots,R_k}\circ 
 \(\prod_{i=1}^k \mu_{A_{i}^1,\dots, A^\ell_{i}}\)\circ \tau \circ
  \(\prod_{j=1}^\ell  \pi_{A_{1}^j,\dots, A_{k}^j}\)
   \circ \pi _{S_1,\dots,S_\ell}(\lambda'')
   \]
   where $A_i^j = R_i \cap S_j$ and $\tau$ is an appropriate twisting bijection. From the associativity and commutativity of $\mu$ and $\pi$, it follows that this formula can be rewritten as \[ \lambda = \mu_{A_1^1,\dots, A_\ell^k} \circ \pi_{A_1^1,\dots,A_\ell^k}(\lambda'')\] where $A_1^1,\dots,A_\ell^k$  are the sets $A_i^j$  listed in any order.
In light of the preceding formula, to show that $(\lambda,\lambda'') \in \{\prec\}[I]$, it remains only to check $\lambda \neq \lambda''$. To this end, note that   $\lambda \in \im(\mu_{A_1^1,\dots, A_\ell^k})$ so also $\lambda \in \im(\mu_{S_1,\dots,S_\ell})$ by the associativity of $\mu$.
It follow by induction on $\ell$ using Lemma \ref{injsurj-fact} that $\pi_{S_1,\dots,S_\ell} \circ \mu_{S_1,\dots,S_\ell}$ is the identity map,
 so $ \mu_{S_1,\dots,S_\ell}\circ \pi_{S_1,\dots,S_\ell}$ acts as the identity map on $\im(\mu_{S_1,\dots,S_\ell})$. From this we deduce that $\lambda \neq \lambda''$, since otherwise we have the contradiction
$  \lambda'=   \mu_{S_1,\dots,S_\ell}\circ \pi _{S_1,\dots,S_\ell}(\lambda'') = \mu_{S_1,\dots,S_\ell}\circ \pi _{S_1,\dots,S_\ell}(\lambda) = \lambda.$

Let $\cO $ be the minimal transitive subspecies of $\P\times\P$ with $(\lambda,\lambda') \in \cO[I]$   whenever  $\lambda = \mu_{S,T}\circ \pi_{S,T} (\lambda')\neq \lambda'$ for some disjoint decomposition $I = S\sqcup T$.
Write $\lambda \leq \lambda'$ if $\lambda=\lambda'$ or $(\lambda,\lambda') \in \cO[I]$.
We have shown that $\{\prec\}$ is a partial order, and so $\{\prec\} \subset \cO$.
To show the reverse containment, we must check that
if
$\lambda' \in \P[I]$
and
 $I = S_1\sqcup \dots \sqcup S_k$ is a pairwise disjoint decomposition
 then
\[ \lambda\leq \lambda'\qquad\text{where}\qquad\lambda\omdef= \mu_{S_1,\dots,S_k} \circ \pi_{S_1,\dots,S_k}(\lambda').\]
We already know that this holds when $k\leq 2$ and in general the desired claim follows by induction on $k$. In detail, assume $k> 2$ and define 
\[\lambda'' = \mu_{S_1\sqcup S_2,S_3,\dots,S_k}\circ \pi_{S_1\sqcup S_2,S_3,\dots,S_k}(\lambda').\]
By induction $\lambda'' \leq \lambda'$, and it follows from Lemma \ref{preorder-lem} by an argument similar to the one in the previous paragraph that $ \lambda = \mu_{S_1,I-S_1} \circ \pi_{S_1,I-S_1}(\lambda'')$, 
so $\lambda \leq \lambda''$. As $\cO$ is transitive, $\lambda \leq \lambda'$, and so we conclude that $\{\prec\} = \cO$.
\end{proof}

\begin{example}
We describe the partial order $\{\prec\}$ when $\P$ is one of the set species $\E$, $\E_C$, $\fk S$, $\L$, or $\Pi$, and  $\mu$ and $\pi$ are  given as in
Examples \ref{mu-ex} and \ref{pi-ex}.
\ben
\item[(i)] Exponential species: if $\P=\E$ then $\{\prec\}$ is the empty subspecies of $\E\times \E$. 

\item[(ii)] Maps: if $\P=\E_C$ then $\{\prec\}$ is again trivial: all distinct $\lambda,\lambda' \in \E_C[I]$ are incomparable.

\item[(iii)] Permutations: if $\P = \fk S$ then $\{\prec\}$ is the partial order such that if  $\lambda,\lambda' \in \fk S[I]$ are permutations 
then  $\lambda \preceq \lambda'$ if and only the cycles of $\lambda'$ are each shuffles of some number of cycles of $\lambda$. Here, a cycle $c$ is a \emph{shuffle} of two cycles $a$ and $b$ if we can write $c=(c_1,\dots, c_n)$ and there are indices $1\leq i_1 < \dots < i_k \leq n$ such that $a = (c_{i_1},\dots, c_{i_k})$ and $b$ is the cycle given by deleting $c_{i_1},\dots,c_{i_k}$ from $c$. For example, there are six  cycles which are shuffles of $(1,2)$ and $(3,4)$:
\[ (1,2,3,4),\quad (1,2,4,3),\quad (1,3,2,4),\quad (1,4,2,3),\quad (1,3,4,2),\quad (1,4,3,2).\]
In turn, a cycle $c$ is a shuffle of $n$ cycles $a_1,\dots,a_n$ if $c$ is a shuffle of $a$ and $b$ where  $a=a_1$ and $b$ is some shuffle of $a_2,\dots,a_n$. (When $n=1$ this occurs precisely when $c=a_1$.)

\item[(iv)] Linear orders: if $\P =\L$ then $\{\prec\}$ is not well-defined since $\mu$ is not commutative.

\item[(v)] Set partitions: if $\P = \Pi$ then $\{\prec\}$ is the refinement partial order.
\een
\end{example}

The following results describe some noteworthy properties of the partial order $\{ \prec\}$.

\begin{lemma}\label{order-lem1}
Assume 
$(\kk\P,\nabla^\mu,\Delta^\pi) \in \cocoHopf(\Sp^\circ)$
and define $\{\prec\}$ as  Theorem \ref{order-thm}.
Let $I = S\sqcup T$ be a disjoint decomposition of a finite set and fix $\alpha \in \P[S]$ and $\beta \in \P[T]$ and $\lambda \in \P[I]$. The following properties then hold:
\ben

\item[(a)] $\mu_{S,T}(\alpha,\beta) \succeq \lambda$ if and only if $\lambda = \mu_{S,T}(\alpha',\beta')$ for some  $\alpha' \preceq \alpha$ and $\beta '\preceq \beta$.

\item[(b)] $\mu_{S,T}(\alpha,\beta) \preceq \lambda  $ if and only if $\pi_{S,T}(\lambda) = (\alpha',\beta')$ for some $\alpha' \succeq \alpha$ and $\beta '\succeq \beta$.

\een
\end{lemma}

\begin{proof}
To prove (a), first suppose $\lambda = \mu_{S,T}(\alpha',\beta')$ for some $\alpha' \preceq \alpha$ and $\beta' \preceq \beta$. We wish to show that $\lambda \preceq \mu_{S,T}(\alpha,\beta)$. For this, we note by definition that there are pairwise disjoint decompositions $S = S_1\sqcup \dots \sqcup S_k$ and $T= T_1 \sqcup \dots \sqcup T_\ell$ such that 
\[ \alpha'
 =  \mu_{S_1,\dots,S_k}\circ \pi_{S_1,\dots,S_k}(\alpha)
 \qquand
 \beta'
 =
 \mu_{T_1,\dots,T_\ell}\circ \pi_{S_1,\dots,S_\ell}(\beta).\]
 Using  the associativity of $\mu$ and $\pi$ and the fact that $ \pi_{S,T}\circ \mu_{S,T}$ is the identity map, it follows    that 
 \[ \lambda =\mu_{S,T}(\alpha',\beta') = \mu_{S_1,\dots,S_k,T_1,\dots,T_\ell} \circ \pi_{S_1,\dots,S_k,T_1,\dots,T_\ell}(\mu_{S,T}(\alpha,\beta))\]
 and thus $\lambda \preceq \mu_{S,T}(\alpha,\beta)$. Conversely suppose $\lambda \preceq \mu_{S,T}(\alpha,\beta)$. We wish to show that $\lambda = \mu_{S,T}(\alpha',\beta')$ for some $\alpha' \preceq \alpha$ and $\beta' \preceq \beta$.
 Let $I = I_1\sqcup \dots \sqcup I_k$ be a pairwise disjoint decomposition such that 
 \[ \lambda = \mu_{I_1,\dots,I_k} \circ \pi _{I_1,\dots,I_k}(\mu_{S,T}(\alpha,\beta))\]
 and define $S_i = I_i \cap S$ and $T_i = I_i \cap T$. By applying Lemma \ref{preorder-lem} to the composition $ \pi _{I_1,\dots,I_k}\circ \mu_{S,T}$ and then using the associativity and  commutativity of $\mu$ and $\pi$, we can rewrite this expression as 
 \[ \lambda = \mu_{S,T} \circ (\mu_{S_1,\dots,S_k} \times \mu_{T_1,\dots,T_k}) \circ (\pi_{S_1,\dots,S_k}\times \pi_{T_1,\dots, T_k})(\alpha,\beta).\]
 It follows that the desired pair $(\alpha',\beta')$ is given by setting $\alpha ' = \mu_{S_1,\dots,S_k}\circ \pi_{S_1,\dots,S_k}(\alpha)$ and $\beta' =   \mu_{T_1,\dots,T_k} \circ  \pi_{T_1,\dots, T_k}(\beta)$. This establishes part (a).
 
 To prove the second part, let $(\alpha',\beta') = \pi_{S,T}(\lambda)$. If $\alpha \preceq \alpha'$ and $\beta\preceq \beta'$ then $\mu_{S,T}(\alpha,\beta) \preceq \mu_{S,T}(\alpha',\beta')$ by part (a), so $\mu_{S,T}(\alpha,\beta) \preceq \lambda$ by transitivity as  $\mu_{S,T}(\alpha',\beta')\preceq \lambda$.
 Conversely, suppose $\mu_{S,T}(\alpha,\beta) \preceq \lambda$. Let $I = I_1 \sqcup \dots \sqcup I_k$ be a pairwise disjoint decomposition such that 
\[\mu_{S,T}(\alpha,\beta) = \mu_{I_1,\dots,I_k} \circ \pi _{I_1,\dots,I_k}(\lambda)\]
and define $S_i = I_i \cap S$ and $T_i = I_i \cap T$. Then $(\alpha,\beta) = \pi_{S,T}\circ  \mu_{I_1,\dots,I_k} \circ \pi _{I_1,\dots,I_k}(\lambda)$, 
 so after
 applying Lemma \ref{preorder-lem} to the composition $ \pi_{S,T}\circ  \mu_{I_1,\dots,I_k}$ and invoking as usual the associativity and commutativity of $\mu$ and $\pi$, it follows that 
\[(\alpha,\beta) =(\mu_{S_1,\dots,S_k}\circ  \mu_{T_1,\dots,T_k}) \circ (\pi _{S_1,\dots,S_k}\times \pi_{T_1,\dots,T_k})(\alpha',\beta')\]
which shows precisely that $\alpha \preceq \alpha'$ and $\beta\preceq \beta'$.
\end{proof}

If $X$ is a labeled set partition then we define its \emph{shape} to be $\sh(X)=\{ B : (B,\lambda) \in X\}$, the unlabeled set partition  whose blocks are the unlabeled blocks of $X$.
In turn, if $\lambda \in \P[I]$ and 
$(\kk\P,\nabla^\mu,\Delta^\pi) \in \cocoHopf(\Sp^\circ)$, then define $\sh(\lambda)$ as the shape of the labeled partition which is the preimage of $\lambda$ under the isomorphism described in Theorem \ref{muFSD-mainthm}.
As a   consequence of the preceding lemma we have this proposition.

\begin{proposition}\label{lattice-cor}
Assume 
$(\kk\P,\nabla^\mu,\Delta^\pi) \in \cocoHopf(\Sp^\circ)$
and define $\{\prec\}$ as in  Theorem \ref{order-thm}.
If $I$ is a finite set and  $\lambda \in \P[I]$, then the partially ordered set
\[\{ \lambda ' \in \P[I] : \lambda' \preceq \lambda\}\]  
is a finite lattice. If $\lambda'$ and $\lambda''$ both belong to this lattice, then $\lambda' \preceq \lambda''$ if and only if $\sh(\lambda') \leq \sh(\lambda'')$ in the ordering of set partitions by refinement.
\end{proposition}

\begin{remark}
While it follows in the situation of the proposition  that the map 
\[\{ \lambda ' \in \P[I] : \lambda' \preceq \lambda\} \to \{ \Lambda \in \Pi[I] : \Lambda \leq \sh(\lambda)\}\] given by $\lambda' \mapsto \sh(\lambda')$  is order-preserving and injective, this map is not necessarily a bijection.
\end{remark}

\begin{proof}
The set $X = \{ \lambda ' \in \P[I] : \lambda' \preceq \lambda\}$ is finite and has a unique maximum given by $\lambda$.
To show that $X$ is a lattice it is enough to show that  every pair of elements $\lambda',\lambda'' \in X$ has a greatest lower bound \cite[Proposition 3.3.1]{Stan1}. To this end, 
let $I = R_1 \sqcup \dots \sqcup R_k = S_1 \sqcup \dots \sqcup S_\ell$  be two pairwise disjoint decompositions such that 
\be\label{thislast}  \lambda' = \mu_{R_1,\dots,R_k}\circ \pi _{R_1,\dots,R_k}(\lambda)
 \qquand
 \lambda''=   \mu_{S_1,\dots,S_\ell}\circ \pi _{S_1,\dots,S_\ell}(\lambda).
\ee
We claim the   greatest lower bound of these elements in $X$   is precisely
\[ \gamma\omdef 
=  
\mu_{A_1,\dots,A_{k\ell}}\circ \pi _{A_1,\dots,A_{k\ell}}(\lambda') 
=
\mu_{A_1,\dots,A_{k\ell}}\circ \pi _{A_1,\dots,A_{k\ell}}(\lambda'') 
\]
where $A_1,\dots, A_{k\ell}$ are the intersections $ R_i \cap S_j$ listed in any order.
The proof of this claim is straightforward using $n$-ary versions of
Theorem \ref{selfcompat-thm} and Lemma \ref{order-lem1}, which hold by induction. We sketch the main idea.
First, Lemma \ref{order-lem1}(a) implies that the set of elements in $X$ which are lower bounds for $\lambda'$ and $\lambda''$ is contained in the images of both
$ \mu_{R_1,\dots,R_k}$ and $\mu_{S_1,\dots,S_\ell}$.
From Theorem \ref{selfcompat-thm}(c), it follows that
 the intersection of these images is just the image of $\mu_{A_1,\dots,A_{k\ell}}$. 
Using Lemma \ref{order-lem1}(b), we deduce that an element in $\im(\mu_{A_1,\dots,A_{k\ell}})$ is a lower bound for $\lambda'$ and for $\lambda''$ if and only if it is the image of a lower bound for $ \pi _{A_1,\dots,A_{k\ell}}(\lambda)=  \pi _{A_1,\dots,A_{k\ell}}(\lambda') =  \pi _{A_1,\dots,A_{k\ell}}(\lambda'')$. Such a lower bound therefore has $\gamma$ as an upper bound by Lemma \ref{order-lem1}(a).

For the last part of the proposition, let $\Lambda = \{ A_1,\dots, A_{k\ell} \} - \{\varnothing\}\in \Pi[I]$ and note that we may assume in \eqref{thislast} that $\sh(\lambda') = \{R_1,\dots,R_k\}$ and $\sh(\lambda'') = \{S_1,\dots,S_\ell\}$. If $\sh(\lambda') \leq \sh(\lambda'')$ then $\Lambda = \sh(\Lambda')$, so $\lambda' = \gamma \preceq \lambda''$ by the (co)unitality of $\mu$ and $\pi$.
Conversely, suppose $\lambda' \preceq \lambda''$, so that $\gamma  = \lambda'$. Since $\gamma$ belongs to the image of $\mu_{A_1,\dots, A_{k\ell}}$, its shape must refine $\Lambda$, so since  $\Lambda \leq \sh(\lambda'')$ by construction,  we have $\sh(\lambda') = \sh(\gamma) \leq \sh(\lambda'')$. 
%
\end{proof}

%
%

The properties     highlighted in Lemma \ref{order-lem1} in some sense characterize all partial orders of the form $\{\prec\}$. In particular, if we are given a partial order on the basis of a strongly self-dual connected Hopf monoid with two certain properties, then this order arises via the construction in Theorem \ref{order-thm} from a unique comultiplicative system. The following theorem makes this idea precise.

\begin{theorem}\label{order-thm2}
Assume $\mu$ is associative, unital, and Hopf self-compatible. Suppose $\{<\}$ is a minimal partial order on $\P$ such that for any disjoint finite sets $S$, $T$ the following properties hold:
 \ben
\item[(A)] If 
 $(\lambda,\lambda') \in \P[S]\times \P[T]$ then $\mu_{S,T}$ induces a poset isomorphism
\[ \{ \alpha \in \P[S] : \alpha \leq \lambda\} \times\{ \alpha' \in \P[T] : \alpha' \leq \lambda'\} \xrightarrow{\sim} \{\gamma \in \P[S\sqcup T] : \gamma \leq \mu_{S,T}(\lambda,\lambda')\}\]
where 
on the left 
we consider  $(\alpha,\alpha') \leq (\beta,\beta')$ if and only if $\alpha\leq\beta$ and $\alpha'\leq\beta'$.

\item[(B)] If 
 $\lambda \in \P[S \sqcup T]$ then   the poset 
$\{ \lambda' \in \im(\mu_{S,T}) : \lambda' \leq \lambda\}$
has a unique maximal element.
 \een
 There is then a unique  comultiplicative system $\pi$ for $\P$ such that
$(\kk\P,\nabla^\mu,\Delta^\pi) \in \cocoHopf(\Sp^\circ)$
and  such that
  $\{<\}$ coincides with the partial order
defined from $\mu$ and $\pi$ by
  Theorem \ref{order-thm}.
\end{theorem}

\begin{remark} 
The order $\{\prec\}$ in Theorem \ref{order-thm} always has   properties (A) and (B): the first property is   equivalent to part (a) of Lemma \ref{order-lem1}, while (B) holds since the same lemma implies that the unique maximal element of  
 $\{ \lambda' \in \im(\mu_{S,T}) : \lambda' \preceq \lambda\}$ is just $\mu_{S,T}\circ \pi_{S,T}(\lambda)$.
\end{remark}

\begin{proof}
Note that $\mu$ is injective and commutative by Theorem \ref{selfcompat-thm}.
Given disjoint sets $S$, $T$, we may therefore define $\pi_{S,T} : \P[S\sqcup T] \to \P[S]\times \P[T]$  as the map such that 
 $\pi_{S,T}(\lambda)$ is the   preimage under $\mu_{S,T}$ of the unique maximal element of $\{ \lambda' \in \im(\mu_{S,T}) : \lambda' \leq \lambda\}$.  
 Since $\{<\}$ is a  subspecies of $\P\times \P$, the family of maps $\pi =(\pi_{S,T})$  is a comultiplicative system for $\P$.  
It is clear from 
 Lemma \ref{order-lem1}(b) that $\pi$ is the only possible comultiplicative system with respect to which the partial order in Theorem \ref{order-thm} could coincide with $\{ <\}$.

 For each finite set $S$ and $\lambda \in \P[S]$ define
$ q_\lambda 
=\sum_{\lambda' \preceq \lambda}{\lambda'} \in \kk\P[S]$ 
where the sum is over  $\lambda' \in \P[S]$.
Observe that if $S$, $T$ are disjoint finite sets and $\alpha \in \P[S]$ and $\beta \in \P[T]$ and $\lambda \in \P[S\sqcup T]$, then 
 \[ \nabla^\mu_{S,T}(q_{\alpha}\otimes q_{\beta}) = q_{\mu_{S,T}(\alpha,\beta)}
 \qquand
 \Delta^\mu_{S,T}(q_\lambda) = q_{\lambda'}\otimes q_{\lambda''}\text{\ \ where\ \ }\pi_{S,T}(\lambda)=(\lambda',\lambda'').\]
Here, the equation on the left follows from property (A), while the right equation is clear from property (B) and the definitions of $\Delta^\mu$ and $\pi$. 
Since $(\kk\P,\nabla^\mu,\Delta^\mu) \in \cocoHopf(\Sp^\circ)$  and since the elements $q_\lambda$ for $\lambda \in \P[S]$ are linearly independent, it is straightforward to deduce from these formulas  that
$\pi$  is coassociative, cocommutative, counital, and Hopf compatible with $\mu$.

Thus $(\kk\P,\nabla^\mu,\Delta^\pi) \in \cocoHopf(\Sp^\circ)$, and  if $S$, $T$ are disjoint finite sets and $\lambda \in \P[S\sqcup T]$ then $\mu_{S,T}\circ \pi _{S,T}(\lambda)$ is the unique maximal element of $\{ \lambda' \in \im(\mu_{S,T}) : \lambda' \leq \lambda\}$. 
Since the partial order $\{\prec\}$ defined with respect to $\mu$ and $\pi$   in Theorem \ref{order-thm}
is the minimal partial order such that $\lambda \preceq \mu_{S,T}\circ \pi_{S,T}(\lambda')$ always holds,
it follows that $\{\prec\} \subset \{<\}$. 
Conversely, since 
both $\{<\}$ and $\{\prec\}$ have properties (A) and (B) and since $\{ <\}$ is assumed to be a minimal order with these properties, we must have $\{ <\} \subset \{\prec\}$.
Hence these orders are equal.
%
%
%
\end{proof}

 \subsection{Explicit isomorphisms}\label{bases-sect}

Maintaining the conventions of the previous section,  we let 
 $\P$ denote a fixed finite set species with a multiplicative system $\mu$ and a comultiplicative system $\pi$. Here, we always assume that $\P[\varnothing]$ is a singleton set, and that $\mu$ and $\pi$ are (co)associative, (co)cocommutative, (co)unital, and Hopf compatible.
 %
Equivalently, we assume $(\kk\P,\nabla^\mu,\Delta^\pi) \in \cocoHopf(\Sp^\circ)$.
 We may then let $\{ \prec\}$ 
denote the partial order on $\P$ defined by Theorem \ref{order-thm} with respect to $\mu$ and $\pi$.

From the data $(\P,\mu,\pi)$ under these hypotheses, our notation naturally allows the definition of three connected Hopf monoids which are linearized, commutative, and cocommutative, namely:
\be\label{hhh-eq} \h = (\kk\P,\nabla^\pi,\Delta^\mu)\qquand \h' = (\kk\P,\nabla^\mu,\Delta^\mu)\qquand \h'' = (\kk\P,\nabla^\mu,\Delta^\pi).\ee
Here $\h''$ is a Hopf monoid by assumption, $\h'$ is a Hopf monoid by Lemma \ref{comcomp-lem}, and $\h$ is a Hopf monoid by Fact \ref{hopfdual-fact}. We do not consider the fourth triple $(\kk \P,\nabla^\pi,\Delta^\pi)$ since by Corollary \ref{deltanabla-cor} this a connected Hopf monoid only if it equal to $\h'$. It is not entirely obvious that these are actually isomorphic Hopf monoids, and in this section we construct explicit isomorphisms demonstrating this fact. 

This amounts to finding   bases of $\kk\P$ in which the formulas for the product and coproduct match those of $\h$, $\h'$, and $\h''$ in turn.
For this purpose, for each finite set $I$ and $\lambda \in \P[I]$, we define
\[
p_\lambda = \sum_{\lambda' \succeq \lambda} {\lambda'}
\qquand
q_\lambda 
=\sum_{\lambda' \preceq \lambda}p_{\lambda'}
\]
where all sums  are over the  elements of $\P[I]$.
To refer to the collections of these elements,  define 
\[
p_\bullet[I] = \{ p_\lambda : \lambda \in \P[I]\}
\qquand
q_\bullet[I] = \{q_\lambda : \lambda \in \P[I]\}
\] 
for finite sets $I$. 
The triangularity of the sums defining $p_\lambda$ and $q_\lambda$ and the fact that $\{\prec\}$ is a partial order on $\P$ imply that  $p_\bullet$ and $q_\bullet$ are then bases (and in particular, set subspecies) of $\kk\P$.
%
%
%
We now present the main result of this section.

\begin{theorem}\label{basis-thm} The connected Hopf monoids $\h$, $\h'$,  and $\h''$ given by \eqref{hhh-eq}
are
 all isomorphic. 
 In particular,  the following holds:
\ben
\item[(a)] The morphism $\kk\P \to \kk\P$ defined by $p_\lambda \mapsto \lambda$ gives an isomorphism 
$ \h \xrightarrow{\sim} \h'.$

\item[(b)] The morphism $\kk\P \to \kk\P$ defined by $q_\lambda \mapsto \lambda$  gives an isomorphism
$ \h \xrightarrow{\sim} \h''.$

\een
\end{theorem}

\begin{proof}

To prove the theorem it suffices to compute the product and coproduct in $\h$ for the bases $p_\bullet$ and $q_\bullet$.
Fix a disjoint decomposition $I = S\sqcup T$  of a finite set.  For part (a) we must show that 
\be\label{id1-eq}
\nabla^\pi_{S,T}(p_\alpha \otimes p_{\beta}) = p_{\mu_{S,T}(\alpha,\beta)}\ee
when $(\alpha,\beta) \in \P[S]\times \P[T]$ and that
 \be\label{id2-eq}
\Delta^\mu_{S,T}(p_\lambda) =\begin{cases}   p_\alpha \otimes p_\beta &\text{if $\lambda = \mu_{S,T}(\alpha,\beta)$ for some $(\alpha,\beta) \in \P[S]\times \P[T]$} \\ 0&\text{otherwise}\end{cases}\ee
when $\lambda \in \P[I]$.
We consider the product first. By definition
\[ \nabla^\pi_{S,T}(p_\alpha \otimes p_{\beta}) =  
\sum_{\substack{\alpha' \in \P[S] \\ \alpha' \succeq \alpha}}  \sum_{\substack{\beta' \in \P[T] \\ \beta' \succeq \beta}}
 \sum_{\substack{\gamma \in \P[S\sqcup T] \\ \pi_{S,T}(\gamma) = (\alpha',\beta')}} \gamma.
\] 
The second part of Lemma \ref{order-lem1} asserts precisely that this triple sum is equal to the sum of $\gamma \in \P[S\sqcup T]$ such that $\gamma \succeq \mu_{S,T}(\alpha,\beta)$, which is  $p_{\mu_{S,T}(\alpha, \beta)}$ as desired. In turn, we compute
\[
 \Delta^\mu_{S,T}(p_\lambda) = 
  \sum_{\substack{\lambda' \in \P[I] \\ \lambda' \succeq \lambda}} 
 \sum_{\substack{(\alpha',\beta') \in \P[S]\times \P[T] \\ \mu_{S,T}(\alpha',\beta')= \lambda'}} \alpha'\otimes \beta'
=
\sum_{\substack{(\alpha',\beta') \in \P[S]\times \P[T] \\ \mu_{S,T}(\alpha',\beta')\succeq \lambda}} \alpha'\otimes \beta'.
 \]
By  the first part of Lemma \ref{order-lem1}, $\alpha'\otimes \beta'$ appears in this sum if and only if $\lambda = \mu_{S,T}(\alpha,\beta)$ for some $\alpha \preceq \alpha ' $ and $\beta \preceq \beta'$. As $\mu_{S,T}$ is injective, the pair $(\alpha,\beta) \in \P[S]\times \P[T]$ is unique if it exists, so if $\lambda \notin \im(\mu_{S,T})$ then the last sum is zero. Alternatively, if $\lambda = \mu_{S,T}(\alpha,\beta)$ for some $(\alpha,\beta)$ then  $ \Delta^\mu_{S,T}(p_\lambda)$ is  the sum of $\alpha'\otimes \beta'$ over all $(\alpha',\beta') \in \P[S]\times \P[T]$ with $\alpha'\succeq \alpha$ and $\beta' \succeq \beta$, which is precisely $p_{\alpha}\otimes p_{\beta} $. 

This establishes (a). To prove (b), we must show that 
\[\nabla^\pi_{S,T}(q_\alpha \otimes q_{\beta}) = q_{\mu_{S,T}(\alpha,\beta)} \qquand \Delta^\mu_{S,T}(q_\lambda) = q_{\lambda'} \otimes q_{\lambda''}
\]
when $(\alpha,\beta) \in \P[S]\times \P[T]$ and $\lambda \in \P[I]$ and  $\pi_{S,T}(\lambda)= (\lambda',\lambda'')$.
The proof of these identities is similar to the arguments above, and follows by combining Lemma \ref{order-lem1} with  \eqref{id1-eq} and \eqref{id2-eq}.
%
\end{proof}

Since $\h'$  is SSD,  the isomorphisms $\h\cong \h'\cong\h''$ imply the following corollary.

\begin{corollary}\label{last-cor} If a connected Hopf monoid or its dual is finite-dimensional,  commutative, cocommutative, and linearized in some basis, then it is strongly self-dual in some other basis.
\end{corollary}

\end{document}